\documentclass[preprint,3p,times,authoryear]{elsarticle}
\usepackage{booktabs}
\usepackage{float}
\usepackage{amsmath}
\usepackage{amsthm}
\usepackage{amsmath,amssymb,amsfonts,mathrsfs,hyperref,microtype, amsthm}
\usepackage{graphicx}
\usepackage{graphicx,mathabx}
\usepackage{enumerate}
\usepackage{subfig}
\graphicspath{ {images/} }
\numberwithin{equation}{section}

\newtheorem{theorem}{Theorem}[section]

\newtheorem{corollary}{Corollary}[section]
\newtheorem{remark}{Remark}
\newcommand{\BS}{\boldsymbol}

\newcommand{\rmnum}[1]{\romannumeral #1}
\newcommand{\Rmnum}[1]{\expandafter\@slowromancap\romannumeral #1@}
\newcommand{\trp}{{\sf T}}

\journal{ }

\makeatletter
\def\ps@pprintTitle{%
   \let\@oddhead\@empty
   \let\@evenhead\@empty
   \def\@oddfoot{\reset@font\hfil\thepage\hfil}
   \let\@evenfoot\@oddfoot
}

\begin{document}

\begin{frontmatter}
\author[]{ Osama Idais\corref{cor1}}
\cortext[cor1]{Corresponding author}
\ead{osama.idais@ovgu.de}
\author[]{Rainer Schwabe}
\ead{rainer.schwabe@ovgu.de}
\address{\small Institute for Mathematical Stochastics,  Otto-von-Guericke-University  Magdeburg,\\ \small PF 4120, D-39016 Magdeburg, Germany}

 \title{Analytic solutions for locally optimal designs \\  for gamma models \\ having linear predictor without intercept}

\begin{abstract}
The gamma model is  a generalized linear model for gamma-distributed outcomes. The model is widely applied in psychology, ecology or medicine. In this paper we focus on gamma models  having a linear predictor without  intercept. For a specific scenario sets of locally D- and A-optimal designs are to be developed.  Recently, \cite{Gaffke2018}  established a complete class and an essentially complete class of designs  for gamma models to obtain locally D-optimal designs. However to extend this approach to gamma model without an intercept term is complicated. To solve that further techniques have to be developed in the current work.   Further, by a suitable transformation between gamma models with and without intercept optimality results may be transferred from one model to the other. Additionally by means of The General Equivalence Theorem optimality can be characterized for multiple regression by a system of polynomial inequalities which can be solved  analytically or by computer algebra. By this necessary and sufficient conditions on the parameter values can be obtained for the local D-optimality of particular designs. The robustness of the derived designs with respect to misspecifications of the initial parameter values is examined by means of their local D-efficiencies.
\end{abstract}


\end{frontmatter}

\section{Introduction}
\label{}

The gamma model is employed for  outcomes that are non-negative, continuous, skewed and heteroscedastic specifically, when the variances are proportional to the square of the means. The gamma model with its canonical link  ( reciprocal ) is appropriate  for many real life data. In ecology and forestry,  \cite{10.1093/forestscience/55.4.310} mentioned that gamma models offers a great potential for many forestry applications and they used gamma models to analyze plant competition.  In medical context,  \cite{Grover2013AnAO} fitted a gamma model  with duration of diabetes as the response variable and predictors as the rate of rise in serum creatinine (SrCr) and number of successes (number of times SrCr values exceed its normal range (1.4 mg/dl)). For a study about air pollution, \cite{Kurtolu2016} employed  a gamma model to analyze nitrogen dioxide concentrations considering some weather factors ( see also  \cite{Chatterjee:1988:SAL:59088}, Section 8.7). In psychological studies, recently,  \cite{Ng2017}  used a gamma model for modeling the relationship between negative automatic thoughts (NAT) and socially prescribed perfectionism (SPP).\par
  Although, the canonical link is frequently employed in the gamma model but there is always  a doubt about the suitable link function for  outcomes. Therefore, a class of link functions might be employed. The common alternative links mostly come  from the Box-Cox family and the power link family (see \cite{atkinson2015designs}). In fact,  the family of power link  functions includes the canonical link therefore it is  a favorite choice for employment in this paper. \par
In the  theory of optimal designs, the information matrix of a generalized linear model depends on the model parameters through the intensity function. Locally optimal designs can be derived through maximizing a specific optimality criterion at certain values of the parameters.  However, although the gamma model is used in many applications, but it has no considerable attention for optimal designs.  Geometric approaches were employed to derive locally D-optimal designs for a gamma model with single factor ( see  \cite{10.2307/2346142}),  with two factors without intercept  (see \cite{Burridge1992}) and for multiple factors (see \cite{10.2307/2336960}). Some of those results were highlighted on by \cite{atkinson2015designs}. Recently, in \cite{Gaffke2018} we provided analytic solutions for optimal designs for gamma models. A complete class and essentially complete class of designs were established under certain assumptions. Therefore, the complexity of deriving optimal designs is reduced and one can only look for the optimal design in those classes. \par

In the present paper gamma models without intercept are considered. Absence of the intercept term yields a difficulty in deriving D- and A-optimal designs. Our main goal is developing various approaches to obtain, mostly, locally D-optimal designs.  This paper is organized as follows. In section 2, the proposed model, the information matrix and the locally optimal design are presented. In section 3, locally D- and A-optimal designs are derived. In section 4,  a two-factor model with interaction is considered for which locally D-optimal designs are derived. The performance of some  derived D-optimal designs are examined in Section 5. Finally, a brief discussion and  conclusions are given in Section 6. 

\section{Model, information and designs}

Let $y_1, ... ,y_n$ be independent gamma-distributed response variables for $n$ experimental units, where the density is given by
\begin{equation}
P(y_i;\kappa,\lambda_i)=\frac{\lambda_i^{\kappa}}{\Gamma\left(\kappa\right)} y_i^{\kappa-1}e^{-\lambda_i y_i }\,\,\,,\kappa,\, \lambda_i,\,y_i  >0,\,\,   (1\le i\le n), \label{gama}
\end{equation}
The shape parameter $\kappa$ of the gamma distribution  is the same for all $y_i$ but the expectations $\mu_i=E(y_i)$ depend on the values $\BS{x}_i$
of a covariate $\BS{x}$. The canonical link obtained from a gamma distribution (\ref{gama}) is reciprocal (inverse) 
\begin{equation*}
 \eta_i=\kappa/\mu_i,\ \ \mbox{where } \ \ \eta_i=\BS{f}^\trp(\BS{x}_i)\BS{\beta},\,\, \ (1\le i\le n),
\end{equation*}
where $\BS{f}=(f_1,\ldots,f_p)^\trp$ is a given $\mathbb{R}^p$-valued function on the experimental region ${\cal X}\subset \mathbb{R}^\nu,\,\nu\ge 1$ with linearly independent component functions $f_1,\ldots,f_p$, and $\BS{\beta}\in\mathbb{R}^p$ is a parameter vector (see \citet{mccullagh1989generalized}, Section 2.2.4).  Here,  the mean-variance function is $v(\mu)=\mu^2$ and the variance of a gamma distribution is thus given by $\mathrm{var}(y)=\kappa^{-1}\mu^2$ with shape parameter $\kappa>0$. Therefore,  the intensity function  at a point $\BS{x}\in \mathcal{X}$ (see \citet{atkinson2015designs})  is given by
 \begin{equation}
 u(\BS{x},\BS{\beta})=\Bigl(\mathrm{var}(y)\,\Big(\frac{{\rm d}\eta}{{\rm d}\mu}\Big)^2\Bigr)^{-1}=\kappa\bigl(\BS{f}^\trp(\BS{x})\BS{\beta}\bigr)^{-2}.
  \end{equation}
Practically,  there are various link functions that are considered to fit gamma observations.   The power link family which is considered throughout  presents the class of link functions  as in \citet{10.2307/2336960}, see also \citet{atkinson2015designs}, Section 2.5, 
\begin{equation}
\eta_i=\mu_i^\rho,\ \ \mbox{where }\ \ \eta_i=\BS{f}^\trp(\BS{x}_i)\BS{\beta}, \ (1\le i\le n).\label{eq4.1}
\end{equation}
The exponent $\rho$ of the power link function is a given nonzero real number. The intensity function under that family reads as
\begin{equation}
u_0(\BS{x},\BS{\beta})=\kappa \rho^{-2}\bigl(\BS{f}^\trp(\BS{x})\BS{\beta}\bigr)^{-2}.\label{eq4.3}
\end{equation}
Gamma-distributed responses are continuous and non-negative and therefore for a given experimental region ${\cal X}$  we assume throughout  that the parameter vector $\BS{\beta}$ satisfies
\begin{equation}
\BS{f}^\trp(\BS{x})\,\BS{\beta}>0\ \mbox{ for all }\ \BS{x}\in\mathcal{X}. \label{eq2-2}
\end{equation}
The Fisher information matrix for a single observation at a point $\BS{x}\in{\cal X}$
under parameter vector $\BS{\beta}$ is given by \
$u_0(\BS{x},\BS{\beta})\,\BS{f}(\BS{x})\,\BS{f}^\trp(\BS{x})$. 
Note that the positive factor $\kappa \rho^{-2}$ is the same for all  $\BS{x}$ and $\BS{\beta}$ and will not affect any design consideration below. We will ignore that factor and consider a normalized version of the Fisher information matrix at $\BS{x}$ and $\BS{\beta}$, 
\begin{equation}
\BS{M}(\BS{x},\BS{\beta})=\bigl(\BS{f}^\trp(\BS{x})\BS{\beta}\bigr)^{-2}\,\BS{f}(\BS{x})\,\BS{f}^\trp(\BS{x}).\label{eq4.4}
\end{equation}        
 We shall make use of approximate designs with finite support on the experimental region $\mathcal{X}$. The approximate design $\xi$ on $\mathcal{X}$ is represented as  
\begin{equation}
\xi=\left \{  \begin{array}{cccc}   \BS{x}_1 &\BS{x}_2&\dots&\BS{x}_m  \\  
 \omega_1 & \omega_2 &\dots&\omega_m \end{array}\right\}, \label{eq2-4}
\end{equation}
where $m\in\mathbb{N}$, $\BS{x}_1,\BS{x}_2, \dots,\BS{x}_m\in\mathcal{X}$ are pairwise distinct points  
and $\omega_1, \omega_2, \dots, \omega_m>0$ with $\sum_{i=1}^{m} \omega_i=1$. 
The set ${\rm supp}(\xi)=\{\BS{x}_1,\BS{x}_2, \dots,\BS{x}_m\}$ is called the support of $\xi$ and 
$\omega_1,\ldots,\omega_m$ are called the weights of $\xi$ (
see \cite{silvey1980optimal}, p.15).  A design $\xi$  is minimally supported if the number of support points is equal to the number of model parameters (i.e., $
m = p$). A minmal-support design which is also called a saturated design will appear frequently  in the current work.  The information matrix of a design $\xi$  at a parameter point $\BS{\beta}$ is defined by
\begin{eqnarray}
\BS{M}(\xi, \BS{\beta})=\sum_{i=1}^{m}\omega_i \BS{M}(\BS{x}_i, \BS{\beta}).\label{eq2-5}
\end{eqnarray}
Another representation of the information matrix (\ref{eq2-5}) can be considered by defining the $m \times p$ design matrix $\BS{F}=[\BS{f}(\BS{\BS{x}}_1),\dots,\BS{f}(\BS{\BS{x}}_m)]^\trp$ and the $m\times m$ weight matrix $\BS{V}=\mathrm{diag}(\omega_iu(\BS{x}_i,\BS{\beta}))_{i=1}^{m}$ and hence, $\BS{M}(\xi, \BS{\beta})=\BS{F}^\trp\BS{V}\BS{F}$. \par
A locally optimal design minimizes a convex criterion function of the information matrix at a given  parameter point $\BS{\beta}$. Denote by ''$\det$'' and ''${\rm tr}$'' the determinant and the trace of a matrix, respectively. We will employ the popular D-criterion and the A-criterion.  More precisely, a design $\xi^*$ is said to be locally D-optimal (at $\BS{\beta}$) if its information matrix $\BS{M}(\xi^*, \BS{\beta})$ at $\BS{\beta}$ is nonsingular and  $ \det\bigl(\BS{M}^{-1}(\xi^*, \BS{\beta})\bigr)=\min_\xi \det\bigl(\BS{M}^{-1}(\xi, \BS{\beta})\bigr)$ where the minimum on the r.h.s. is taken over all designs $\xi$ whose information matrix at $\BS{\beta}$ is nonsingular. Similarly, a design $\xi^*$ is said to be locally A-optimal (at $\BS{\beta}$) if its information matrix at $\BS{\beta}$ is nonsingular and  ${\rm tr}\bigl(\BS{M}^{-1}(\xi^*, \BS{\beta})\bigr)=\min_\xi {\rm tr}\bigl(\BS{M}^{-1}(\xi, \BS{\beta})\bigr)$ where, again, the minimum is taken over all designs $\xi$ whose information matrix at $\BS{\beta}$ is nonsingular.
\begin{remark}  \rm
It is worthwhile mentioning that the set of designs for which the information matrix is nonsingular does not depend on $\BS{\beta}$ (when $u(\BS{x},\BS{\beta})$ is strictly positive). In particular it is just the set of designs for which the information matrix is nonsingular in the corresponding ordinary regression model (ignoring the intensity $u(\BS{x},\BS{\beta})$).  That is the singularity depends on the support points of a design $\xi$ because its information matrix  $\BS{M}(\xi,\BS{\beta})=\BS{F}^\trp\BS{V}\BS{F}$ is full rank if and only if  $\BS{F}$ is full rank. 
\end{remark}

\begin{remark} \label{rem-1}\rm
If the experimental region is a compact set and the functions $\BS{f}(\BS{x})$ and $u(\BS{x},\BS{\beta})$ are continuous in $\BS{x}$ then the set of all nonnegative definite  information matrices is compact. Therefore,  there exists a locally D- resp. A-optimal design for any given parameter point $\BS{\beta}$.
\end{remark}
In order to verify the  local optimality of  a design The General Equivalence Theorem is usually employed. It  provides necessary and sufficient conditions for a design to be optimal with respect to the optimality criterion, in specific D- and A-criteria and thus the optimality of a suggested design can easily be verified or disproved (see \cite{silvey1980optimal}, p.40, p.48 and  p.54)). The most  generic  one is the celebrated  Kiefer-Wolfowitz equivalence theorem under  D-criterion ( see \citet{kiefer_wolfowitz_1960} ). In the following we  obtain  equivalent characterizations of locally D- and A-optimal designs.

\begin{theorem} \label{theo2-1}Let $\BS{\beta}$ be a given parameter point and let $\xi^*$ be a design 
with nonsingular information matrix $\BS{M}(\xi^*, \BS{\beta})$.\\[1ex]
(a)  The design $\xi^*$ is locally D-optimal (at $\BS{\beta}$) if and only if
\begin{eqnarray*}
&&u(\BS{x},\BS{\beta})\,\BS{f}^\trp(\BS{x})\BS{M}^{-1}(\xi^*, \BS{\beta})
\BS{f}(\BS{x})\leq p \ \ \mbox{ for all }\, \BS{x} \in \mathcal{X}.
\end{eqnarray*}
(b) The design $\xi^*$ is locally A-optimal (at $\BS{\beta}$) if and only if
\begin{eqnarray*}
&&u(\BS{x},\BS{\beta})\,\BS{f}^\trp(\BS{x})\BS{M}^{-2}(\xi^*, \BS{\beta})\BS{f}(\BS{x})\leq 
\mathrm{tr}\bigl(\BS{M}^{-1}(\xi^*, \BS{\beta})\bigr)\ \ \mbox{ for all }\, 
\BS  x \in \mathcal{X}.
\end{eqnarray*}
\end{theorem}
\begin{remark}\rm
The inequalities given by part (a) and part (b) of Theorem \ref{theo2-1} are equations at support points of any D- or A-optimal design, respectively.
\end{remark}

 Throughout, we consider  gamma models that do not explicitly  involve a constant (intercept) term.  More precisely, we  assume that $f_j\neq 1$  for all  ($1\le j\le p$) and thus $f_j(\BS{0})=\BS{0}$ for all ($1\le j\le p$). In particular, we restrict to a first order model with 
\begin{equation}
 \BS{f}(\BS{x})=\BS{x},\,\, \mbox{ where } \,\, \BS{x}=(x_1,\dots,x_\nu)^\trp,\,\,\nu \ge2, \label{eq2-9}
\end{equation}
 and the two-factor model with interaction 
\begin{equation}
  \BS{f}(\BS{x})=(x_1,x_2,x_1x_2)^\trp. \label{eq2-10}
\end{equation}
Surely, condition (\ref{eq2-2}), i.e., $\BS{f}^\trp(\BS{x})\BS{\beta}>0$ for all $\BS{x}\in \mathcal{X}$ implies that $\BS{0}\notin \mathcal{X}$. Therefore, an  experimental region as $\mathcal{X}=[0,\infty)^\nu\setminus\{\BS{0}\}$ is considered.  Note that this experimental region is no longer compact therefore the existence of optimal designs is not assured and has to be checked separately. \par
In contrast, we often consider a compact experimental region  that is a $\nu$-dimensional hypercube 
\begin{equation}
\mathcal{X}=\bigl[a,b\bigr]^{\nu},\,\nu\ge 2 \mbox { with } a,b\in\mathbb{R} \mbox { and } 0<a<b, \label{eq3-1}
\end{equation}
 with  vertices $\BS{v}_i,\,i=1,\dots, K,\,K=2^{\nu}$ given by the points whose $i$-th coordinates are either $a$ or $b$ for all $i=1,\dots,\nu$. \par
In \cite{Gaffke2018}, we showed that under gamma models with regression function $\BS{f}(\BS{x})$   from (\ref{eq2-9}) or  (\ref{eq2-10}) and  experimental region $\mathcal{X}=[a,b]^\nu,\nu\ge2,\,0<a<b$ the design that has support  only among the vertices is at least good as any design that has no support points from the vertices w.r.t the Loewner semi-ordering of information matrices or, more generally, of nonnegative definite $p\times p$ matrices. That is if $\BS{A}$ and $\BS{B}$ are nonnegative definite $p\times p$ matrices we write $A\le B$ if and only if $B-A$ is nonnegative definite.  The set of all designs $\xi$ such that $\mathrm{supp}(\xi)\subseteq\{\BS{v}_1,\dots,\BS{v}_K\}$ is a locally essentially complete class of designs at a given $\BS{\beta}$. As a result, there exists a design $\xi^*$ that is only supported by  vertices of $\mathcal{X}$ which is locally optimal (at $\BS{\beta}$) w.r.t. D- or A-criterion.  On that basis, throughout,  we restrict to designs whose support is a subset of the vertices of $\mathcal{X}$ given by a hypercube  (\ref{eq3-1}).\par

\begin{remark}\label{rem-4}\rm
Let us denote by  $\psi(\BS{x})$ the left hand side of The Equivalence Theorems, Theorem \ref{theo2-1}, part (a) or part (b). Typically $\psi(\BS{x})$ is called the  sensitivity function. Actually, under  non-intercept gamma models $\psi(\BS{x})$  is invariant with respect to simultaneous scale transformation of $\BS{x}$, i.e., $\psi(\lambda \BS{x}) = \psi(\BS{x})$ for any $\lambda>0$. This essentially comes from the fact that the function $\BS{f}_{\BS{\beta}}(\BS{x})=\big(\BS{f}^\trp(\BS{x})\BS{\beta}\big)^{-1}\BS{f}(\BS{x})$ is invariant with respect to simultaneous rescaling of the components of $\BS{x}$, i.e., $\BS{f}_{\BS{\beta}}(\lambda \BS{x})=\BS{f}_{\BS{\beta}}(\BS{x})$. This property is explicitly transferred to the information matrix (\ref{eq4.4}) since it can be represented in form $\BS{M}(\BS{x},\BS{\beta})=\BS{f}_{\BS{\beta}}(\BS{x})\BS{f}_{\BS{\beta}}^\trp(\BS{x})$, and hence $\BS{M}(\lambda\BS{x},\BS{\beta})=\BS{M}(\BS{x},\BS{\beta})$.  In fact, this property plays a main rule in the solution of the forthcoming optimal designs. 
\end{remark}
\section{First order gamma model}

In this section we consider a gamma model with  
\begin{eqnarray}
&&\BS{f}(\BS{x})=(x_1,\dots,x_\nu)^\trp,\,\nu\ge 2,\,\,\BS{x}\in \mathcal{X},  \\
&&\mbox { where } \nonumber\\
&&\BS{f}_{\BS{\beta}}(\BS{x})=\frac{1}{\beta_1x_1+\dots+\beta_\nu x_\nu}\,\left({\begin{array}{c} x_1 \\ \vdots\\ x_\nu\end{array}}\right). \label{eq_f}
\end{eqnarray}
Firstly let  the experimental region $\mathcal{X}=[0,\infty)^\nu\setminus\{\BS{0}\}$ be considered. Denote by $\BS{e}_i$ for all $(1\le i \le \nu)$ the $\nu$-dimensional unit vectors. The parameter space is determined by  condition (\ref{eq2-2}), i.e., $\BS{x}^\trp\BS{\beta}>0$ for all $\BS{x}\in \mathcal{X}$ which implies that   $\BS{\beta}\in (0,\infty)^\nu$, i.e.,  $\beta_i>0$ for all ($1 \le i \le \nu$). 
Let the induced experimental region is given by $\BS{f}_{\BS{\beta}}(\mathcal{X})=\{\BS{f}_{\BS{\beta}}(\BS{x}): \BS{x}\in \mathcal{X}\}$. Although $\mathcal{X}$ is not compact  but $\BS{f}_{\BS{\beta}}(\mathcal{X})$ is compact. That is 
\[
\BS{f}_{\BS{\beta}}(\mathcal{X})= \mathrm{Conv}\{\BS{f}_{\BS{\beta}}(\BS{e}_i):\BS{e}_i \in \mathcal{X}, i=1,\dots, \nu\},
\]
 where ‘Conv’ denotes convex hull operation. That means each point $\BS{f}_{\BS{\beta}}(\BS{x})$ for all $\BS{x}\in \mathcal{X}$ can be written as a convex combination of $ \BS{f}_{\BS{\beta}}(\BS{e}_i)$  for all ($1\le i \le \nu$), i.e., we obtain $\BS{f}_{\BS{\beta}}(\BS{x})=\sum_{i=1}^{\nu}\alpha_i\BS{f}_{\BS{\beta}}(\BS{e}_i)$ for some $\alpha_i\ge 0$ for all ($1\le i \le \nu$) such that $\sum_{i=1}^{\nu}\alpha_i=1$ (here, $\alpha_i=\beta_ix_i/\sum_{i=1}^{\nu}\beta_ix_i, i=1,\dots,\nu$). As a consequence, the set of all nonnegative definite  information matrices is compact and  a locally optimal design can be obtained (cp.  Remark \ref{rem-1}). 

\begin{theorem} \label{theo3-1}
Consider the experimental region $\mathcal{X}=[0,\infty)^\nu\setminus\{\BS{0}\}$.  Let $\BS{x}_i^*=\BS{e}_i$\,\,for all $(1\le i\le \nu)$. Given a parameter point $\BS{\beta}$. Then
\begin{enumerate}[(i)]
\item The saturated design $\xi^*$ that assigns equal weight  $\nu^{-1}$ to the  support $\BS{x}_i^*$\,for all\, $(1\le i\le \nu)$   is locally D-optimal (at $\BS{\beta}$).
\item The saturated design $\zeta^*$  that assigns the weights $\omega_i^*=\beta_i/\sum_{i=1}^\nu\beta_i$\, for all\,$(1\le i \le \nu)$  to the corresponding design point $\BS{x}_i^*$\, for all $(1\le i\le \nu)$ is locally A-optimal  (at $\BS{\beta}$).
\end{enumerate}
 \end{theorem}

\begin{proof} Define  the $\nu\times \nu$ design matrix $\BS{F}=\mathrm{diag}(\BS{e}_i)_{i=1}^\nu$ with $\nu\times \nu$ weight matrix $\BS{V}=\mathrm{diag}(\omega_i^*/\beta_{i}^{2})_{i=1}^\nu$.  Then we have $\BS{F}^\trp\BS{V}\BS{F}=\mathrm{diag}(\omega_i^*/\beta_{i}^{2})_{i=1}^{\nu}$ and $\big(\BS{F}^\trp\BS{V}\BS{F}\big)^{-1}=\mathrm{diag}(\beta_{i}^{2}/\omega_i^{*})_{i=1}^{\nu}$ where; \vspace{-3ex}
\begin{align*}\vspace{-3ex}
&\mbox{For D-optimality, }\,\,\omega_i^{*}=\nu^{-1} \forall i,\,\BS{M}^{-1}\bigl(\xi^*, \BS{\beta}\bigr)=\nu\,\mathrm{diag}(\beta_i^{2})_{i=1}^\nu\mbox{ and }\,\BS{f}^\trp(\BS{x})\mathrm{diag}(\beta_i^{2})_{i=1}^\nu\BS{f}(\BS{x})=\sum_{i=1}^{\nu}\beta_{i}^{2}x_i^2.\\[-2ex]
&\mbox{For A-optimality, }\,\,\omega_i^*=\beta_i/\sum_{i=1}^\nu\beta_i\,\forall i,\,\,\BS{M}^{-1}\bigl(\zeta^*, \BS{\beta}\bigr)=\big(\sum\limits_{i=1}^{\nu}\beta_{i}\big)\mathrm{diag}(\beta_{i})_{i=1}^{\nu},\,\,\mathrm{tr}\big(\BS{M}^{-1}(\zeta^*, \BS{\beta})\big) =\big(\sum_{i=1}^{\nu}\beta_{i}\big)^2,\\[-2ex]
&\BS{M}^{-2}\bigl(\zeta^*, \BS{\beta}\bigr)=\big(\sum_{i=1}^{\nu}\beta_{i}\big)^2\mathrm{diag}(\beta_{i}^2)_{i=1}^{\nu},\,\mbox{ and }\BS{f}^\trp(\BS{x})\big(\sum_{i=1}^{\nu}\beta_{i}\big)^2\mathrm{diag}(\beta_{i}^2)_{i=1}^{\nu}\BS{f}(\BS{x})=\big(\sum_{i=1}^{\nu}\beta_{i}\big)^2\sum_{i=1}^{\nu}\beta_{i}^2 x_i^2.\\[-6ex]
\end{align*}
Hence,  by The Equivalence Theorem (Theorem \ref{theo2-1}, part (a) and  part (b)) $\xi^*$ resp. $\zeta^*$ is locally D- resp. A-optimal (at $\BS{\beta}$) if  and only if 
$\,\,  \big(\sum_{i=1}^{\nu}\beta_ix_i\big)^{-2}\big(\sum_{i=1}^{\nu}\beta_{i}^{2}x_i^2\big)\leq 1$ for all \,$\BS{x} \in \mathcal{X}$  which is equivalent to  $-2\sum_{i<j=1}^{\nu}\beta_i\beta_jx_ix_j\le 0$\, for all \,$\BS{x} \in \mathcal{X}$. The latter inequality holds true by model assumptions $\beta_i>0, x_i \ge 0$ for all ($1\le i \le \nu$).
 \end{proof}

\begin{remark}\rm
The locally D-optimal designs provided by part (i) of Theorem \ref{theo3-1} is robust against misspecified values of the model parameter in its parameter space $(0,\infty)^\nu$.
\end{remark}

While the information matrix is invariant w.r.t. to simultaneous rescaling of the components of $\BS{x}$ as it is mentioned in Remark \ref{rem-4}, the result of  Theorem \ref{theo3-1} can be extended: 

\begin{corollary} \label{theo5.0.1}
Consider the experimental region $\mathcal{X}=[0,\infty)^\nu\setminus\{\BS{0}\}$. Given a constant vector $\BS{a}=(a_1,\dots,a_\nu)^\trp$\, such that $a_i>0$\, and\, $a_i\BS{e}_i\in \mathcal{X}$\,\,for all $(1\le i\le \nu)$.  Let $\BS{x}_i^*=a_i\BS{e}_i$\,\,for all $(1\le i\le \nu)$. Given a parameter point $\BS{\beta}$. Then
\begin{enumerate}[(i)]
\item The saturated design $\xi_{\BS{a}}^*$ that assigns equal weight  $\nu^{-1}$ to the  support $\BS{x}_i^*$\, $(1\le i\le \nu)$   is locally D-optimal (at $\BS{\beta}$).
\item The saturated design $\zeta_{\BS{a}}^*$  that assigns the weights $\omega_i^*=\beta_i/\sum_{i=1}^\nu\beta_i$\, for all\,$(1\le i \le \nu)$  to the corresponding design point $\BS{x}_i^*$\, for all $(1\le i\le \nu)$ is locally A-optimal  (at $\BS{\beta}$).
\end{enumerate}
 \end{corollary}

Actually, the derived optimal designs $\xi_{\BS{a}}^*$ and $\zeta_{\BS{a}}^*$ are not unique at a given parameter point $\BS{\beta}$. The convex combinations of locally optimal designs is optimal w.r.t. D- or A-criterion.   In the following we introduce a set of locally D-optimal designs and a set of locally A-optimal designs.

\begin{corollary} \label{theo3.5.1.} Under assumptions of Corollary \ref{theo5.0.1}. Let $\xi_{\BS{a}}^*$ and  $\zeta_{\BS{a}}^*$ are locally D- and A-optimal design at $\BS{\beta}$, respectively. Let 
\begin{align*}
\Xi^*&=\mathrm{Conv}\{ \xi_{\BS{a}}^*:\BS{a}=(a_1,\dots,a_\nu)^\trp, a_i>0\,\, \forall i=1,\dots,\nu\}.\\        
\mathrm{Z}^*&=\mathrm{Conv}\{ \zeta_{\BS{a}}^*:\BS{a}=(a_1,\dots,a_\nu)^\trp, a_i>0\,\, \forall i=1,\dots,\nu\}.
\end{align*}
 Then $\Xi^*$  is a set of  locally D-optimal designs (at $\BS{\beta}$) and \, $\mathrm{Z}^*$  is a set of  locally A-optimal designs (at $\BS{\beta}$).
 \end{corollary}

In what follows we consider a hypercube $\mathcal{X}=[a,b]^\nu,\nu\ge2,\,0<a<b$, as an experimental region.  As given in Remark \ref{rem-4}, we have  $\BS{f}_{\BS{\beta}}(\lambda \BS{x})=\BS{f}_{\BS{\beta}}(\BS{x}),\,\lambda>0$ and thus a transformation of a gamma model without intercept to a gamma model with intercept can be obtained if, in particular, $\lambda=x_1^{-1},\,x_1>0$. This reduction is useful to determine precisely the candidate support points of a design. Another reduction might be obtained on the parameter space  when, in particular,  $\lambda=\beta_1^{-1}$. \par

Let us begin with the simplest case  $\nu=2$. A transformation of a two-factor model without intercept to a single-factor model with intercept is employed. Based on that D- and A-optimal designs are derived. 

\begin{theorem}\label{The3-2}
Consider the experimental region $\mathcal{X}=[a,b]^2,\,0<a<b$.  Let $\BS{x}_1^*=(a,b)^\trp$ and $\BS{x}_2^*=(b,a)^\trp$. Let $\BS{\beta}=(\beta_1,\beta_2)^\trp$ be given such that $\BS{\beta}^\trp\BS{x}_i^*>0$ for all  $i=1,2$ (which is equivalent to  condition (\ref{eq2-2})). Then, the unique locally D-optimal design $\xi_{\rm\scriptsize D}^*$ (at $\BS{\beta}$) is the two-point design supported by $\BS{x}_1^*$ and $\BS{x}_2^*$ with equal weights $1/2$. The unique locally A-optimal design $\xi_{\rm\scriptsize A}^*$ (at $\BS{\beta}$)
is the two-point design supported by $\BS{x}_1^*$ and $\BS{x}_2^*$ with weights $\omega_1^*=\frac{\beta_1b+\beta_2a}{(\beta_1+\beta_2)(a+b)}$ and  $\omega_2^*=\frac{\beta_1a+\beta_2b}{(\beta_1+\beta_2)(a+b)}$.
\end{theorem}

\begin{proof} Since  $\BS{f}_{\BS{\beta}}(x_1^{-1}\, \BS{x})=\BS{f}_{\BS{\beta}}(\BS{x})$ for all $\BS{x}=(x_1,x_2)^\trp\in[a,b]^2$, we write 
\begin{eqnarray*}
&&\BS{f}_{\BS{\beta}}(\BS{x})=\bigl(\beta_1x_1+\beta_2x_2\bigr)^{-1}\,\bigl(x_1\,,\, x_2\bigr)^\trp\,=\,
\bigl(\beta_1+\beta_2t\bigr)^{-1}\,\bigl(1\,,\,t\bigr)^\trp,\\
&&\mbox{where }\ t=t(\BS{x})=x_2/x_1.
\end{eqnarray*}
 So the information matrices coincide with those from a single-factor gamma model with intercept.  The range of $t=t(\BS{x})$, as $\BS{x}$ ranges over $[a,b]^2$ is the interval $\bigl[(a/b)\,,\,(b/a)\bigr]$. Note also that the end points $a/b$ and $b/a$ come from the unique points $\BS{x}_1^*=(a,b)^\trp$ and $\BS{x}_2^{*}=(b,a)^\trp$, respectively. Following the proof of Theorem 4.1 in  \cite{Gaffke2018}) yields the stated results on the locally D- and A-optimal designs in our theorem, where for local A-optimality we get 
\[
\omega_1^*=\frac{\bigl(\beta_1+\beta_2\frac{a}{b}\bigr)\sqrt{1+(\frac{b}{a})^2}}{
\bigl(\beta_1+\beta_2\frac{a}{b}\bigr)\sqrt{1+(\frac{b}{a})^2} + 
 \bigl(\beta_1+\beta_2\frac{b}{a}\bigr)\sqrt{1+(\frac{a}{b})^2}}
\]
and it is  straightforwardly to verify that the above quantity is equal to $\frac{\beta_1b+\beta_2a}{(\beta_1+\beta_2)(a+b)}$. 
\end{proof} 

\begin{remark}\label{rem1} \rm
Actually, in case $\nu \ge 3$ an analogous transformation of the model as in the proof of Theorem \ref{The3-2} is obvious,
\begin{eqnarray*}
&&\BS{f}_{\BS{\beta}}(\BS{x})=\bigl(\beta_1+\beta_2 t_1+\beta_3 t_2+\ldots+\beta_\nu t_{\nu-1}\bigr)^{-1}
\bigl(1,t_1,\ldots,t_{\nu-1}\bigr)^\trp,\\
&&\mbox{where }\ t_j=t_j(\BS{x})=x_{j+1}/x_1 \,\mbox{ for all }\,(1\le j\le\nu-1)\ \mbox{ for }\BS{x}=(x_1,x_2,\ldots,x_\nu)^\trp\in[a,b]^\nu, 0 <a<b,
\end{eqnarray*}
leading thus to a first order model with intercept employing a $(\nu-1)$-dimensional factor $\BS{t}=(t_1,\ldots,t_{\nu-1})^\trp$.
However, its range $\bigl\{\BS{t}(\BS{x})\,:\,\BS{x}\in[a,b]^\nu\bigr\}\subseteq\mathbb{R}^{\nu-1}$ 
is not a cube but a more complicated polytope.  
E.g., for $\nu=3$ it can be shown that
\begin{eqnarray*}
\Bigl\{\BS{t}(\BS{x})\,:\,\BS{x}\in[a,b]^3\Bigr\}\,=\,{\rm Conv}\biggl\{
\left({\small\begin{array}{c} a/b \\ 1\end{array}}\right),
\left({\small\begin{array}{c} 1\\ a/b\end{array}}\right),
\left({\small\begin{array}{c} a/b\\ a/b \end{array}}\right),
\left({\small\begin{array}{c} b/a \\ 1\end{array}}\right),
\left({\small\begin{array}{c} 1 \\ b/a \end{array}}\right),
\left({\small\begin{array}{c} b/a \\ b/a \end{array}}\right)
\biggr\}
\end{eqnarray*}
where for each $\BS{x}\in[a,b]^3$ we get $\BS{t}(\BS{x})=(x_2/x_1,x_3/x_1)^\trp$ as it is depicted in  Figure \ref{fig:cubepoly1} for, in specific,  $a=1$ and $b=2$.  In \cite{Gaffke2018} we showed that the support of a design is a subset of vertices of the ploytope. One notes that for each vertex $\BS{v}\in\{(a,a,a)^\trp,(b,b,b)^\trp\}$ we get $\BS{t}(\BS{v})=(1,1)^\trp$ which lies in the interior of the convex hull above, i.e., $(1,1)^\trp$ is a proper convex combination of the vertices of the polytope.  Thus  this reduction on the vertices implies that both vertices $(a,a,a)^\trp$ and $(b,b,b)^\trp$ of the hupercube $[a,b]^3$ are out of consideration as support points of any optimal design. 
\end{remark}

\begin{figure}[H]
    \centering
    \subfloat[ The experimental region $\mathcal{X}={[}1,2{]}^3$. ]{{\includegraphics[width=5cm, height=5cm]{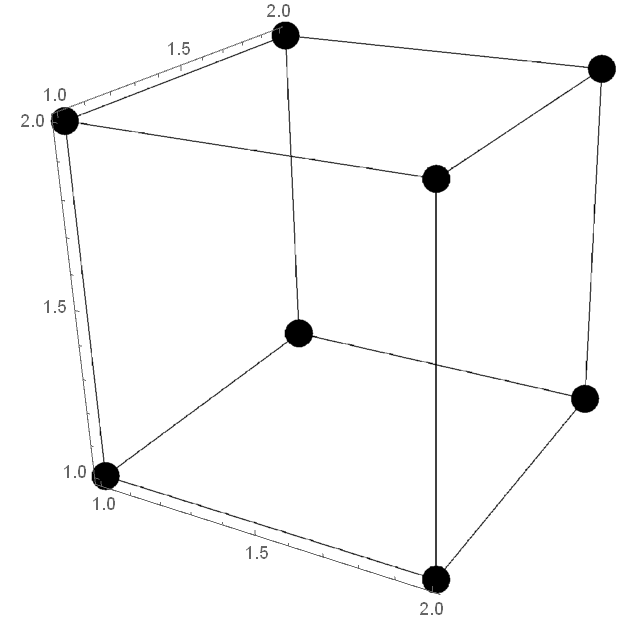} }}
    \qquad
    \subfloat[The transformed experimental region ${\rm Conv}\bigl\{(1/2,1/2)^\trp,(1/2,1)^\trp,(1,1/2)^\trp,(2,1)^\trp,(1,2)^\trp,(2,2)^\trp\bigr\}$. The interior point is $(1,1)^\trp$ which represents the original points $(1,1,1)^\trp$ and $(2,2,2)^\trp$ in ${[}1,2{]}^3$.]{{\includegraphics[width=5cm, height=5cm]{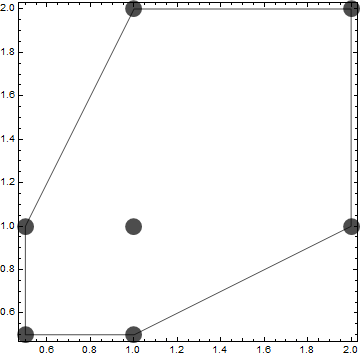} }}
    \caption{ }%
    \label{fig:cubepoly1}
\end{figure}

Let us concentrate on the  experimental region $\mathcal{X}=[1,2]^3$. The linear predictor of a three-factor gamma model is given by $\eta(\BS{x},\BS{\beta})=\beta_1x_1+\beta_2x_2+\beta_3x_3$. Assume that $\beta_2=\beta_3=\beta$, so the set of all parameter points under condition (\ref{eq2-2}), i.e.,  $\beta_1x_1+\beta_2x_2+\beta_3x_3>0$ for all $\BS{x}=(x_1,x_2,x_3)^\trp\in \mathcal{X}$ is characterized  by  
\[
\beta _1\leq 0,\,  \beta >-\beta _1\, \mbox{ or }\, \beta _1>0,\, \beta>-\frac{1}{4}\beta _1
\]
which is shown by Panel (a) of Figure \ref{fig:opt1}. \par
 Let the vertices of $\mathcal{X}=[1,2]^3$ be denoted by $\BS{v}_1=\big(1,1,1\big)^\trp$, $\BS{v}_2=\big(2,1,1\big)^\trp$,  $\BS{v}_3=\big(1,2,1\big)^\trp$, $\BS{v}_4=\big(1,1,2\big)^\trp$, $\BS{v}_5=\big(1,2,2\big)^\trp$, $\BS{v}_6=\big(2,1,2\big)^\trp$, $\BS{v}_7=\big(2,2,1\big)^\trp$, $\BS{v}_8=\big(2,2,2\big)^\trp$ with intensities $u_i=u(\BS{v}_i,\BS{\beta}),\, i=1,\dots, \nu$.   Actually, the region shown in Figure \ref{fig:opt1} is the parameter space of $\BS{\beta}=(\beta_1,\beta_2,\beta_3)^\trp$ only when $\beta_2=\beta_3$. We aim at finding  locally D-optimal designs at a given parameter point in that space. The expression `` optimality subregion'' will be used  to refer to  a subset of parameter points  where  saturated designs or  non-saturated designs with similar support  are locally D-optimal.  

In the next theorem we introduce  the locally D-optimal designs on respective optimality subregions. Table \ref{T-22} presents the order of the intensities in all optimality subregions and the corresponding D-optimal designs. The intensities for both vertices $\BS{v}_1$ and $\BS{v}_8$ are ignored due to the reduction (cp. Remark \ref{rem1}). It is noted that on each  subregion the  vertices of highest intensities perform mostly as a support of the corresponding D-optimal design. In particular,  analytic solution of  the locally D-optimal designs of type $\xi_5^*$ at a point $\BS{\beta}$ from the subregion  $-3\beta _1<\beta <-\frac{6}{5}\beta _1$,\,\, $\beta _1< 0$ cannot be developed so that numerical results are to be derived (cp. Remark \ref{rem-rr}).

\begin{table}[H]
\centering
{
\scalebox{0.8}{
\resizebox{\textwidth}{!}{
\begin{tabular}{lllll }
 \hline
 Subregions&\,\,\,\,\,\ \ Intensities order &\,\,\,\,\,\ \ D-optimal design \\
 \hline
$\beta>0$,\,$\beta _1= 0$   & $ u_2>u_3=u_4=u_6=u_7>  u_5$  &\hspace{6ex}$\xi_1^*$ \\[0.5ex]
  $\beta \ge -3\beta _1$,\,\,$\beta _1< 0$         &	$ u_2>u_6=u_7\approx u_3=u_4> u_5$ &\hspace{6ex}$\xi_1^*$\\[0.5ex]
  $\beta >\frac{1}{5}\beta _1$,\,\,$\beta _1>0$         &$ u_2>u_3=u_4>u_6=u_7>  u_5$&\hspace{6ex}$\xi_1^*$ \\[0.5ex]
$-\frac{1}{4}\beta _1<\beta \le -\frac{5}{23}\beta _1$,\,\,$\beta _1>0$       &	$ u_5>u_3=u_4> u_6=u_7> u_2$&\hspace{6ex}$\xi_2^*$\\[0.5ex]
$-\frac{5}{23}\beta _1<\beta < \frac{1}{5}\beta _1$,\,\,$\beta _1>0$           &$ u_3=u_4\ge u_5>u_2\ge  u_6=u_7$&\hspace{6ex}$\xi_3^*$\\[0.5ex]
 $-\beta _1< \beta\le -\frac{6}{5}\beta_1$,\,\,$\beta _1<0$& $ u_2> u_6=u_7>  u_3=u_4>u_5$&\hspace{6ex}$\xi_4^*$\\[0.5ex]
$-3\beta _1<\beta <-\frac{6}{5}\beta _1$,\,\, $\beta _1< 0$ & $ u_2> u_6=u_7>  u_3=u_4>u_5$ &\hspace{6ex}$\xi_5^*$\\
 \hline
\end{tabular}}}
}
\caption{The order of intensity values according to subregions correspond to D-optimal designs} 
\label{T-22}
\end{table}

\begin{figure}[H]
    \centering
    \subfloat[]{{\includegraphics[width=7cm, height=6cm]{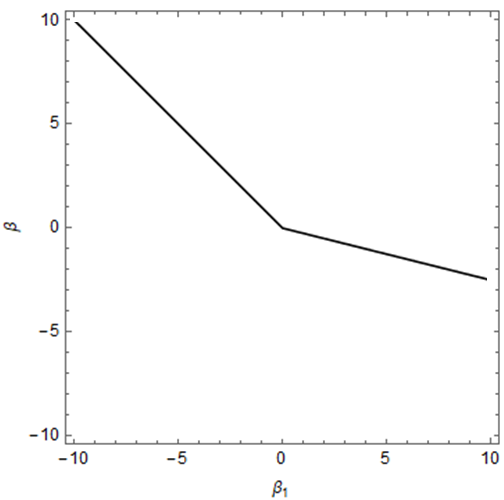} }}
    \qquad
    \subfloat[]{{\includegraphics[width=7cm, height=6cm]{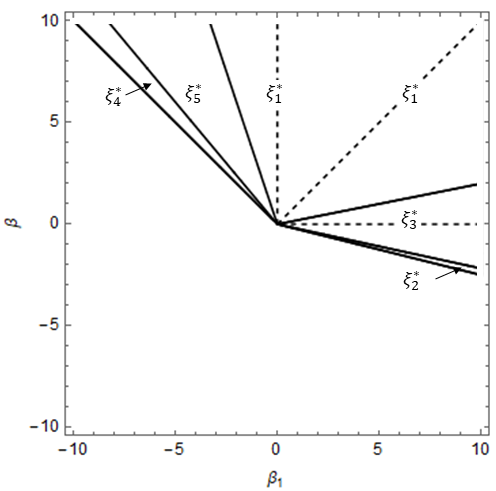} }}
    \caption{Panel (a):  The parameter space of $\BS{\beta}=(\beta_1,\beta_2,\beta_3)^\trp$ such that $\beta_2=\beta_3=\beta$.\\ Panel (b): Dependence of locally D-optimal designs from  Theorem \ref{The5-5} on $\BS{\beta}=(\beta_1,\beta_2,\beta_3)^\trp$ such that $\,\beta_2=\beta_3=\beta$. The dashed lines are; diagonal: $\beta=\beta_1$, vertical: $\beta_1=0$, horizontal: $\beta=0$.}%
    \label{fig:opt1}
\end{figure}

\begin{theorem} \label{The5-5}
Consider  the experimental region $\mathcal{X}=[1,2]^3$. Let a parameter point $\BS{\beta}=(\beta_1,\beta_2,\beta_3)^\trp$ be given such that  $ \beta_2=\beta_3=\beta$\, with \, $\beta >-\beta _1$\, if \,$\beta _1\le 0$\, or  \, $\beta >-\frac{1}{4}\beta _1$\, if\, $\beta _1>0$.    Then the following  designs are locally D-optimal (at $\BS{\beta}$).  
\begin{enumerate}[(i)]
\item If \,\,  $\beta>0$,\,$\beta _1= 0$\,\, or\,\,  $\beta \ge -3\beta _1$,\,$\beta _1< 0$\,\, or  \,\,  $\beta >\frac{1}{5}\beta _1$,\,$\beta _1>0$  then 
\begin{eqnarray*}
\xi_1^*=\left(\begin{array}{ccc}\BS{v}_2& \BS{v}_3&\BS{v}_4\\[.5ex]
\frac{1}{3}&\frac{1}{3}&\frac{1}{3}\end{array}\right).
\end{eqnarray*}
\item  If \,\ $-\frac{1}{4}\beta _1<\beta \le -\frac{5}{23}\beta _1$,\,\,$\beta _1>0$  then
\begin{eqnarray*}
\xi_2^*=\left(\begin{array}{ccc}\BS{v}_3& \BS{v}_4&\BS{v}_5\\[.5ex]
\frac{1}{3}&\frac{1}{3}&\frac{1}{3}\end{array}\right).
\end{eqnarray*}
\item If    $-\frac{5}{23}\beta _1<\beta < \frac{1}{5}\beta _1$,\,\,$\beta _1>0$  then 
\begin{eqnarray*}
\xi_3^*=\left(\begin{array}{cccc}\BS{v}_2& \BS{v}_3&\BS{v}_4&\BS{v}_5\\[.5ex]
\omega_1^*&\omega_2^*&\omega_3^*&\omega_4^*\end{array}\right).
\end{eqnarray*}
where 
\begin{equation*}
\omega_1^*=\frac{5+23\,\gamma}{16\,(1+4\,\gamma)},\,\,
\omega_2^*=\omega_3^*=\frac{9\,(1+3\gamma)^2}{32\,(1+\gamma)(1+4\,\gamma)},\,\, \omega_4^*=\frac{1-\gamma-20\,\gamma^2}{8\,(1+\gamma)(1+4\,\gamma)},\,\, \gamma=\frac{\beta}{\beta_1}.
 \end{equation*}
 \item If    $-\beta _1< \beta\le -\frac{6}{5}\beta_1$,\,\,$\beta _1<0$  then 
 \begin{eqnarray*}
\xi_4^*=\left(\begin{array}{ccc}\BS{v}_2& \BS{v}_6&\BS{v}_7\\[.5ex]
\frac{1}{3}&\frac{1}{3}&\frac{1}{3}\end{array}\right).
\end{eqnarray*}
\end{enumerate}
\end{theorem}

\begin{proof}
The proof is obtained by making use of the condition of the Equivalence Theorem (Theorem \ref{theo2-1}, part (a)).  So that we develop a system of feasible inequalities evaluated at the vertices $\BS{v}_i$ for all $(1\le i \le 8)$. For simplicity in computations when $\beta_1\neq0$ we utilize  the ratio $\gamma=\beta/\beta_1$  of which the range is given by $(-\infty,-1)\cup (-\frac{1}{4},\infty)$. It turned out that some inequalities are equivalent and thus a resulted system is reduced to an equivalent system of a few  inequalities.   The intersection of the set of  solutions of each system with the range of $\gamma$  leads to the optimality condition (subregion) of the corresponding optimal design.  Note that  $u_{1}=\beta_1^{-2}\big(1\,+2\,\gamma\big)^{-2}$, $u_{2}=\beta_1^{-2}\big(2\,+2\,\gamma\big)^{-2}$, $u_{3}=u_{4}=\beta_1^{-2}\big(1+3\,\gamma\big)^{-2}$, $u_{5}=\beta_1^{-2}\big(1+4\,\gamma\big)^{-2}$,  $u_{6}=u_{7}=\beta_1^{-2}\big(2+3\,\gamma\big)^{-2}$, $u_{8}=\beta_1^{-2}\big(2+4\,\gamma\big)^{-2}$. \\
($\rmnum{1}$)  The $3\times3$ design matrix $\BS{F}=[\BS{v}_2, \BS{v}_3,\BS{v}_4]^\trp$ is given by
\begin{eqnarray*}
\BS{F}=\left(
\begin{array}{ccc}
 2 & 1 & 1 \\\noalign{\medskip}
 1 & 2 & 1 \\\noalign{\medskip}
 1 & 1 & 2 \\\noalign{\medskip}
\end{array}
\right)\,\, \mbox{ with }\,\, \BS{F}^{-1}= \left(
\begin{array}{rrr}
 \frac{3}{4} & -\frac{1}{4} & -\frac{1}{4} \\\noalign{\medskip}
 -\frac{1}{4} & \frac{3}{4} & -\frac{1}{4} \\\noalign{\medskip}
 -\frac{1}{4} & -\frac{1}{4} & \frac{3}{4} \\\noalign{\medskip}
\end{array}
\right) \,\,\mbox{ and weight matrix } \BS{V}=\mathrm{diag}\big(u_{2},u_{3},u_{4}\big).
\end {eqnarray*}
Hence, the condition of The Equivalence Theorem is given by
\begin{align}
&\BS{f}^\trp(\BS{x})\BS{F}^{-1} \BS{V}^{-1}\big(\BS{F}^\trp\big)^{-1}\BS{f}(\BS{x})\leq \big(\beta_1x_1+\beta_2x_2+\beta_3x_3\big)^2\,\,\,\,\,\,\forall \BS{x}
\in \{1,2\}^3. \label{eq-eq}
\end{align}
For case    $\beta>0$,\,$\beta _1= 0$,    condition (\ref{eq-eq}) is  equivalent to 
\begin{align*}
4\big(3x_1-(x_2+x_3)\big)^2+9\bigl(\big( 3x_2-(x_1+x_3)\big)^2+\big(3x_3-(x_1+x_2)\big)^2\bigr)\le 16\big(x_2+x_3\big)^2\,\,\forall \BS{x}\in \{1,2\}^3,
\end{align*}
which is  independent of $\beta$ and is satisfied by $\BS{v}_i$ for all $(1\le i \le 8)$ with equality holds for the support. 
For the other cases condition (\ref{eq-eq}) is  equivalent to 
\begin{align}
&\big(3x_1-(x_2+x_3)\big)^2(2\,+2\,\gamma)^{2}+\bigl(\big( 3x_2-(x_1+x_3)\big)^2\nonumber\\
&+\big(3x_3-(x_1+x_2)\big)^2\bigr)(1+3\,\gamma\big)^{2}\le 16\big(x_1+\gamma(x_2+x_3)\big)^2\,\,\forall \BS{x}\in \{1,2\}^3. \label{eq33-2}
\end{align}
After some lengthy but straightforward calculations,  the above inequalities reduce to 
\begin{align}
&15 \gamma^2+2\gamma-1 \ge 0\label{eq3-4} \mbox {  for } \BS{v}_5 \\
&3 \gamma^2+10\gamma+3 \ge 0\label{eq3-5}  \mbox {  for $\BS{v}_6 $ or $\BS{v}_7$ } 
\end{align}
The l.h.s. of each of  (\ref{eq3-4}) and (\ref{eq3-5})  above is a polynomial in $\gamma$ of degree 2 and thus the sets of solutions are given by  $(-\infty,-\frac{1}{3}]\cup[\frac{1}{5},\infty)$ and $(-\infty,-3]\cup[-\frac{1}{3},\infty)$, respectively. Note that the interior bounds are the roots of the respective polynomials. Hence, by considering the intersection of both sets  with the range of $\gamma$,  the design $\xi_1^*$ is locally D-optimal if   $\gamma \in (-\infty,-3]\cup[\frac{1}{5},\infty)$ which is equivalent to the optimality subregion $\beta \ge -3\beta _1$,\,$\beta _1< 0$\,\, or  \,\,  $\beta >\frac{1}{5}\beta _1$,\,$\beta _1>0$ given in part ($\rmnum{1}$) of the theorem.\\
($\rmnum{2}$)   The $3\times3$ design matrix $\BS{F}=[\BS{v}_3, \BS{v}_4,\BS{v}_5]^\trp$ is given by
\begin{eqnarray*}
\BS{F}= \left(
 \begin {array}{ccc}
  1&2&1\\ \noalign{\medskip}
  1&1&2\\ \noalign{\medskip}
  1&2&2
  \end {array}
\right) \,\, \mbox{ with }\,\, \BS{F}^{-1}= \left(
 \begin {array}{rrr} 
 2&2&-3\\ \noalign{\medskip}
 0&-1&1\\ \noalign{\medskip}
 -1&0&1
 \end {array}
 \right) \,\, \mbox{ and weight matrix }\,\,\BS{V}=\mathrm{diag}\big(u_{3},u_{4},u_{5}\big).
\end {eqnarray*}
Hence, the condition of The Equivalence Theorem is equivalent to
\begin{align*}
&\big( \big(2\,x_1-x_2\big)^2+\big(2\,x_1-x_3\big)^2  \big)\,\big(1+3\,\gamma\big)^{2}+ \big( x_3+x_2-3\,x_1\big)^2\,\big(1+4\,\gamma\big)^{2}\le\big(x_1+\gamma(x_2+x_3)\big)^2\,\,\forall \BS{x}\in \{1,2\}^3, 
\end{align*}
and also  the above inequalities reduce to 
\begin{align}
&69 \gamma^2+38\gamma+5 \le 0 \mbox {  for $\BS{v}_2 $ }.  \label{eq3-7}
\end{align}
Again, the set of solutions of  the polynomial  determined by the l.h.s. of  inequality (\ref{eq3-7}) is given by $[-\frac{1}{3},-\frac{5}{23}]$. By considering the intersection with the range of $\gamma$, the design $\xi_2^*$ is locally D-optimal if $\gamma \in (-\frac{1}{4},-\frac{5}{23}]$.\ \\
($\rmnum{3}$) Consider design $\xi_3^*$. Note that  $\omega_1^*>0$ for all $\gamma>-5/23$, $\omega_2^*>0$ for all $\gamma\in \mathbb{R}$ and $\omega_4^*>0$ for all $\gamma \in (-\frac{1}{4},\frac{1}{5})$, and thus it is obvious  that $\omega_1^*,\omega_2^*,\omega_4^*$ are positive over  $(-\frac{5}{23},\frac{1}{5})$ and $\sum_{i=1}^4 \omega_i^*=1$. The $4\times3$ design matrix is given by $\BS{F}=[\BS{v}_2, \BS{v}_3,\BS{v}_4,\BS{v}_5]^\trp$ with weight matrix $\BS{V}=\mathrm{diag}\big(s_{2},s_{3},s_{4},s_{5}\big)$ where $s_{i}=\omega_{i}^*u_{i}, i=2,3,4,5$ and $s_{3}=s_{4}$.  The information matrix is given by
\begin{eqnarray*}
\BS{M}\bigl(\xi_3^*, \BS{\beta}\bigr)=\left(
\begin {array}{ccc}
 4\,s_{2}+2\,s_{3}+s_{5}&2\,s_{2}+3\,s_{3}+2\,s_{5}&2\,s_{2}+3\,s_{3}+2\,s_{5}\\ \noalign{\medskip}
2\,s_{2}+3\,s_{2}+2\,s_{5}&s_{2}+5\,s_{3}+4\,s_{5}&s_{2}+4\,s_{3}+4\,s_{5}\\ \noalign{\medskip}
2\,s_{2}+3\,s_{3}+2\,s_{5}&s_{2}+4\,s_{3}+4\,s_{4}&s_{2}+5\,s_{3}+4\,s_{5}
\end {array}
 \right) 
\end {eqnarray*}
 and one calculates $\det\BS{M}\bigl(\xi_3^*, \BS{\beta}\bigr)=16\,s_{2}\,s_{3}^{2}+18\,s_{2}\,s_{3}\,s_{5}+s_{3}^{2}s_{5}$. Define the following quantities 
\begin{eqnarray*}
&&c_{1}=\frac{s_{3}(2\,s_{2}+9\,s_{3}+8\,s_{5})}{16\,s_{2}\,s_{3}^{2}+18\,s_{2}\,s_{3}\,s_{5}+s_{3}^{2}s_{5}},\,\,
 c_{2}=\frac{-s_{3}(2\,s_{2}+3\,s_{3}+2\,s_{5})}{16\,s_{2}\,s_{3}^{2}+18\,s_{2}\,s_{3}\,s_{5}+s_{3}^{2}s_{5}},\\
&&c_{3}=\frac{10\,s_{2}\,s_{3}+9\,s_{2}\,s_{5}+s_{3}^{2}+s_{3}\,s_{5}}{16\,s_{2}\,s_{3}^{2}+18\,s_{2}\,s_{3}\,s_{5}+s_{3}^{2}s_{5}},\,\,c_{4}=\frac{-6\,s_{2}\,s_{3}+9\,s_{2}\,s_{5}-s_{3}^{2}}{16\,s_{2}\,s_{3}^{2}+18\,s_{2}\,s_{3}\,s_{5}+s_{3}^{2}s_{5}}.\\
&&\mbox{The inverse of the information matrix is given by }\\
&&\BS{M}^{-1}\bigl(\xi_3^*, \BS{\beta}\bigr)=\left( 
\begin {array}{ccc}
 c_1&c_2&c_2\\ \noalign{\medskip}
 c_2&c_3&c_4\\ \noalign{\medskip}
 c_2&c_4&c_3
 \end {array}
 \right). \,\, \mbox{ Hence, the condition of The Equivalence Theorem}\\
&& \mbox{ is equivalent to}\\
 &&c_1\,x_1^2+c_3\,(x_2^2+x_3^2)+2\,c_2\,(x_1\,x_2+x_1\,x_3)+2\,c_4\,x_2\,x_3 \le 3\,\big(x_1+\gamma\,(x_2+x_3)\big)^2\,\, \forall\,\,\BS{x}\in \{1,2\}^3\\
 &&\mbox{which is equivalent to the following system of inequalities }
 \end{eqnarray*} \vspace{-5ex} 
  \begin{align*}
&c_1+4 c_2+2c_3+2 c_4\leq 3 \left(1+2 \gamma\right)^2 \mbox{ for $\BS{v}_1$ or $\BS{v}_8$ }\\
&4 c_1+12 c_2+5 c_3+4 c_4\leq 3 \left(2+3 \gamma\right)^2 \mbox{ for $\BS{v}_6$ or $\BS{v}_7$ }
\end{align*}
However, due to the complexity of the system above  we employed  computer algebra using  Wolfram Mathematica 11.3 (see \cite{Mathematica}) to obtain the solution for $\gamma$. \\
($\rmnum{4}$) The $3\times3$ design matrix $\BS{F}=[\BS{v}_2, \BS{v}_6,\BS{v}_7]^\trp$ is given by
\begin{eqnarray*}
\BS{F}= \left( 
\begin {array}{ccc} 
2&1&1\\ \noalign{\medskip}
2&1&2\\ \noalign{\medskip}
2&2&1
\end {array}
 \right) \,\, \mbox{ with }\,\, \BS{F}^{-1}= \left( 
 \begin {array}{rrr} 
 \frac{3}{2}&-\frac{1}{2}&-\frac{1}{2}\\ \noalign{\medskip}
 -\frac{1}{2}&0&1\\ \noalign{\medskip}
 -\frac{1}{2}&1&0
 \end {array}
 \right) \,\, \mbox{ and weight matrix }\,\,\BS{V}=\mathrm{diag}\big(u_{2},u_{6},u_{7}\big). 
\end {eqnarray*}
Hence, the condition of The Equivalence Theorem is equivalent to 
\begin{align*}
&\left(\left(x_2-\frac{x_1}{2}\right)^2+ \left(x_3-\frac{x_1}{2}\right)^2\right)\left( 2+ 3 \gamma\right)^2
+ \left(\frac{3 x_1}{2}-x_2-x_3\right)^2\left( 2+2 \gamma\right)^2\le\big(x_1+\gamma(x_2+x_3)\big)^2\,\,\forall \BS{x}\in \{1,2\}^3,
\end{align*}
and the above inequalities reduce to 
\begin{align}
&90\, \gamma ^2+168 \gamma +72\leq 0 \mbox{ for $\BS{v}_3$ or $\BS{v}_4$} \label{eqqq}\\
&6 \gamma ^2+16 \gamma +8\leq 0 \mbox{ for $\BS{v}_8$ }\label{eqqqq}
\end{align}
In analogy to parts ($\rmnum{1}$) and ($\rmnum{2}$) the sets of solutions of  (\ref{eqqq}) and (\ref{eqqqq}) are given by  $[-1.2,-\frac{2}{3}]$ and $[-2,-\frac{2}{3}]$, respectively where the interior bounds are the roots of the respective polynomials. Hence, by considering the intersection of both sets  with the range of $\gamma$,  the design $\xi_4^*$ is locally D-optimal if   $\gamma \in [-1.2,-1)$.
\end{proof}

In Panel (b) of Figure \ref{fig:opt1} the  optimality subregions of $\xi_1^*$, $\xi_2^*$,  $\xi_3^*$ and  $\xi_4^*$ form Theorem \ref{The5-5} are depicted.  Note that  each design of $\xi_1^*$,   $\xi_2^*$ and $\xi_4^*$ denotes a single design whereas $\xi_3^*$ determines a certain type of designs with weights  depend on the parameter values. A well known form of $\xi_3^*$ is obtained at $\beta=(-1/7)\beta_1$ which represents the uniform design on the vertices $\BS{v}_2, \BS{v}_3, \BS{v}_4, \BS{v}_5$.  Additionally, along  the horizontal dashed line, i.e.,  $\beta=0$,  $\xi_3^*$ assigns the weights $\omega_1^*=5/16,\omega_2^*=\omega_3^*=9/32, \omega_4^*=1/8$ to  $\BS{v}_2, \BS{v}_3,\BS{v}_4,\BS{v}_5$, respectively.  For equally size of parameters, i.e., $\beta_1=\beta$ the diagonal dashed line in Panel (b) represents a case  where $\xi_1^*$ is D-optimal.  \par

\begin{remark}\rm \label{rem-rr} Deriving a locally D-optimal design at a given parameter point from the subregion   $-3\beta _1<\beta <-\frac{6}{5}\beta _1$,\,\, $\beta _1< 0$ is  not available analytically. Therefore, employing the multiplicative algorithm (see \cite{yu2010monotonic} and \cite{Harman2009}) in  the software package $\textbf{\textsf{R}}$  (see \cite{R}) provides numerical solutions which show that the locally D-optimal design on that subregion is of form
 \begin{eqnarray*}
\xi_5^*=\left(\begin{array}{ccccc}\BS{v}_2& \BS{v}_3&\BS{v}_4&\BS{v}_6&\BS{v}_7\\[.5ex]
\omega_1^*&\omega_2^*&\omega_2^*&\omega_3^*&\omega_3^*\end{array}\right)
\end{eqnarray*}
which is supported by five vertices  with weights may depend on $\BS{\beta}$.  The equal weights are due to the symmetry. Table {\ref{T-2}} shows some numerical results in terms of the ratio $\gamma=\beta/\beta_1$ where $\gamma\in (-3,-6/5)$  .
\end{remark}

\begin{table}[H]
\centering
{
\scalebox{0.6}{
\resizebox{\textwidth}{!}{
\begin{tabular}{lllllllll }
 \hline
 $\,\,\,\,\,\gamma$&\,\,\,\,\,\ \ $\BS{v}_2$  &\,\,\,\,\,\ \ $\BS{v}_3$ &\,\,\,\,\,\ \ $\BS{v}_4$&\,\,\,\,\,\ \ $\BS{v}_6$&\,\,\,\,\,\ \ $\BS{v}_7$\\
 \hline
 $-2.9$         &0.3312&0.3285&	0.3285 &0.0059  &  0.0059  \\
 $-2.5$         &0.3225&0.3051& 0.3051& 0.0336& 0.0336  \\
$-2$         &	0.3125 &0.2604  &  0.2604  & 0.0833&	0.0833\\
$-1.5$        &	0.3125 &0.1701  &  0.1701  & 0.1736&	0.1736\\
 $-1.23$           &0.3297&0.0325&0.0325&0.3027&0.3027\\
 \hline
\end{tabular}}}
}
\caption{D-optimal designs on $\mathcal{X}=\left[1,2\right]^3$ at $\gamma\in (-3,-6/5)$ where $\gamma=\beta/\beta_1$ and   $-3\beta _1<\beta <-\frac{6}{5}\beta _1$,\,\, $\beta _1< 0$.} 
\label{T-2}
\end{table}

In general, for gamma models without intercept, finding optimal designs for a model with multiple factors, i.e., $\nu> 3$ is not an easy task. The optimal design given by part ($\rmnum{1}$)  of Theorem \ref{The5-5} might be extended for arbitrary number of factors under sufficient and necessarily condition on the parameter points:
\begin{theorem} \label{The3-3}
Consider  the  experimental region $\mathcal{X}=\big[a,b\big]^\nu, \nu \ge 3,\,0<a<b$.   Let $\BS{\beta}$ be a parameter point such that $\BS{f}^\trp(\BS{x})\BS{\beta}>0$ for all $\BS{x}\in \mathcal{X}$. Define  $T(\BS{x})=\sum_{i=1}^{\nu}x_i$, $q=\frac{a}{(\nu-1)a+b}$ and $c_j=(b-a)\beta_j+a\sum_{i=1}^{\nu}\beta_i$\,($1 \le j \le \nu$).  Then the design 
 $\xi^*$ which assigns equal weights $\nu^{-1}$ to the support
  \[ 
   \BS{x}^*_1=\big(b,a,\dots,a\big)^\trp,\, \BS{x}^*_2=\big(a,b,\dots,a\big)^\trp, \,\dots\,,\, \BS{x}^*_\nu=\big(a,a,\dots,b\big)^\trp
   \] 
is locally D-optimal (at $\BS{\beta}$) if and only if for all $ \BS{x}=(x_1,\dots,x_\nu)^\trp\in \{a,b\}^\nu$
\begin{equation}
\sum_{j=1}^{\nu}\big(x_j-qT(\BS{x})\big)^2c_j^{2}\leq (b-a)^2\big(\sum_{j=1}^{v}\beta_jx_j\big)^2. \label{eq4.36}
\end{equation}
\end{theorem}
\begin{proof} Define the $\nu \times \nu$ design matrix  $\BS{F}=[\BS{f}(\BS{x}^*_1),\dots,\BS{f}(\BS{x}^*_\nu)]^\trp$. Thus $
\BS {F}=(b-a)\BS{I}+a\BS{1}\BS{1}^\trp$ and $\BS{F}^{-1}=\frac{1}{(b-a)}\left(
\BS{I}-q\BS{1}\BS{1}^\trp\right)$ where $\BS{I}$ is the $\nu\times \nu$ identity matrix and $\BS{1}$ is a $\nu\times 1$ vector of ones. The information matrix of $\xi^*$  is given by $\BS{M}\bigl(\xi^*, \BS{\beta}\bigr)=\frac{1}{\nu}\BS{F}^\trp \BS{V}\BS{F}$ where $\BS{V}=\mathrm{diag}\Big(u(\BS{x}^*_j,\BS{\beta})\Big)_{j=1}^\nu$ is the $\nu\times \nu $ weight matrix. Note that $u(\BS{x}^*_j,\BS{\beta})=c_j^{-2}$\, for all  \,$(1 \le j \le \nu)$.
Hence, the l.h.s. of the condition of the Equivalence Theorem (Theorem \ref{theo2-1}, part (a)) is equal to  
\begin{align}
&\big(\sum_{j=1}^{v}\beta_jx_j\big)^{-2}\BS{f}^\trp(\BS{x})\BS{M}^{-1}\bigl(\xi^*, \BS{\beta}\bigr)\BS{f}(\BS{x})=\nu\big(\sum_{j=1}^{v}\beta_jx_j\big)^{-2}\BS{f}^\trp(\BS{x})\BS{F}^{-1} \BS{V}^{-1}\BS{F}^{-1}\BS{f}(\BS{x})\nonumber\\
&=\nu\big((b-a)\sum_{j=1}^{v}\beta_jx_j\big)^{-2}\left(
\BS{f}^\trp(\BS{x})-qT(\BS{x})\BS{1}^\trp\right) \mathrm{diag}\Big(c_j^{2}\Big)_{j=1}^\nu\left(
\BS{f}(\BS{x})-qT(\BS{x})\BS{1}\right)\nonumber\\
&=\nu\Big((b-a)\sum_{j=1}^{v}\beta_jx_j\Big)^{-2}\sum_{j=1}^{\nu}\big(x_j-qT(\BS{x})\big)^2c_j^{2}.\label{eq-cond}\\
&\mbox{By Equivalence Theorem design $\xi^*$ is locally D-optimal  if and only if (\ref{eq-cond}) is less than or equal to $\nu$}\nonumber\\
&\mbox {for all \,$\BS{x}\in \{a,b\}^\nu$ leading  resulting inequalities that are equivalent to assumption (\ref{eq4.36}).\nonumber } 
\end{align}
\end{proof}

Note that the D-optimal design given in part ($\rmnum{1}$) of Theorem \ref{The5-5} is a special case of Theorem \ref{The3-3}
 when $\nu=3$ where condition (\ref{eq4.36}) covers condition (\ref{eq33-2}) in the proof of part ($\rmnum{1}$)  of Theorem \ref{The5-5}. Actually it can be seen that already in the general case of Theorem \ref{The3-3} the optimality condition (\ref{eq4.36}) depends only on the ratios $\beta_j/(\sum_{i=1}^{\nu}\beta_i)$ for all ($1\le j \le \nu$). Hence the scaling factor vanishes. Similarly  note that already condition (\ref{eq4.36}) depends   on $a$ and $b$ only through their ratio $a/b$.
However, assuming  the model parameters are having equal size implies that the D-optimality of a design is independent of the model parameters whereas it depends on the ratio $a/b$ as it is shown in the next corollary.

\begin{corollary} \label{The3-4}
Consider the  experimental region $\mathcal{X}=\big[a,b\big]^\nu, \nu \ge 3,\,0<a<b$.  Let $\BS{\beta}$ be a parameter point such that $\beta_j=\beta_{j'}=\beta >0\,\,(1 \le j< {j'} \le \nu)$.  Then the design 
 $\xi^*$ which assigns equal weights $\nu^{-1}$ to the support $\BS{x}^*_1=\big(b,a,\dots,a\big)^\trp$, $\BS{x}^*_2=\big(a,b,\dots,a\big)^\trp$, $\dots$, $\BS{x}^*_\nu=\big(a,a,\dots,b\big)^\trp$ is locally D-optimal (at $\BS{\beta}$) if and only if 
 \begin{eqnarray}
\Big(\frac{b}{a}\Big)^2\ge \frac{\big(\nu-1 \big)\big(\nu-2\big)}{2}. \label{eq3.2}
\end{eqnarray}
\end{corollary}
\begin{proof}
Let $\beta_j=\beta_{j'}=\beta \,\,(1 \le j< {j'} \le \nu)$ then condition (\ref{eq4.36}) of Theorem \ref{The3-3} reduces to  
\begin{eqnarray}
\left((\nu-1)a^2+b^2\right) \left(\sum_{j=1}^{\nu}x_j\right)^2-\left((\nu-1)a+b\right)^2\sum_{j=1}^{\nu}x_j^2\ge 0\,\,\,\, \forall \BS{x}\in \{a,b\}^\nu. \label{eq4.37}
\end{eqnarray}
For  $\BS{x}=(x_1,\dots,x_\nu) \in \{a,b\}^\nu$,  let $r=r(\BS{x})\in\{0,1,\dots, \nu\}$ denote the number of coordinates of $\BS{x}$ that are equal to $b$. Then $\sum_{j=1}^{\nu}x_j^2=\left(\nu-r\right)a^2+r\,b^2$ and $\left(\sum_{j=1}^{\nu}x_j\right)^2=\left(\left(\nu-r\right)a+r\,b\right)^2$. 
Hence, condition (\ref{eq4.37}) is equivalent to
\begin{equation}
(a-b)^2\,\tau\,r^2+(a-b)( (b+a) -2\,a\,\nu\,\tau)\,r+\nu\,a^2(\nu\,\tau-1)\, \ge 0\,\,\forall r \in\{0,1,\dots, \nu\},\, \nu \ge 2  \label{eq4.38}
\end{equation}
where  $\tau=\frac{\left(\nu-1\right)a^2+b^2}{\left(\left(\nu-1\right)a+b\right)^2}$.  The l.h.s. of  inequality (\ref{eq4.38}) is a  polynomial in $r$ of degree 2 with positive leading term. The polynomial attains $0$ at $r=1$ ($r_1=1$ indicates the support of $\xi^*$) and at $r_2=\frac{\nu\,\left(\nu-1\right)\,a^2}{\left(\nu-1\right)a^2+b^2}$. Note that  the polynomial is positive and increasing for all $r>2$ (i,e., (\ref{eq4.38}) holds true ) when  $r_2\le 2$ or, equivalently, $\frac{\nu\,\left(\nu-1\right)\,a^2}{\left(\nu-1\right)a^2+b^2}\le 2$  which coincides with condition (\ref{eq3.2}). 
\end{proof}

\begin{remark} \rm Actually, condition (\ref{eq3.2}) is obviously fulfilled for $\nu=2$ (compare Theorem \ref{The3-2}). For the case $\nu=3$ the bound of l.h.s. of condition (\ref{eq3.2}) is $1$ and, hence always fulfilled.
\end{remark}
\section{ Gamma models with interaction }
In this section we are still dealing with a model without intercept. We consider a model with two factors and with an interaction term where $\BS{f}(\BS{x})=\big(x_1,x_2,x_1x_2 \big)^\trp$ and $\BS{\beta}=\big(\beta_1,\beta_2,\beta_3\big)^\trp$. The experimental region is given by $\mathcal{X}=[a,b]^2,\,0<a<b$ and we aim at deriving a locally D-optimal design.  Our approach is  employing a transformation of the proposed model to a model with intercept by removing the interaction term $x_1x_2$. It follows that 
\begin{align}
\BS{f}_{\BS{\beta}}(\BS{x})&=\bigl(\beta_1x_1 +\beta_2x_2 +\beta_3x_1x_2\bigr)^{-1}\,\big(x_1,x_2,x_1x_2 \big)^\trp \nonumber \\
&=\bigl(\beta_1t_2 +\beta_2 t_1 +\beta_3\bigr)^{-1}\,\big(t_2,t_1,1 \big)^\trp=\BS{f}_{\BS{\beta}}^\circ(\BS{t})  \label{eq4-1}
\end{align}
where $\BS{t}=\bigl(t_{1},t_2\bigr)^\trp, t_j=1/x_j,\, j=1, 2$. The range of $\BS{t}=\BS{t}(\BS{x})$, as $\BS{x}$ ranges over $\mathcal {X}=[a,b]^2$ is a cube given by $ \mathcal T=\bigl[(1/b)\,,\,(1/a)\bigr]^2$.
One can  rearrange the terms of (\ref{eq4-1}) by making use of  the $3 \times 3$ anti-diagonal transformation matrix $\BS{Q}$. That is $ \BS{\tilde{f}}_{\tilde{\BS{\beta}}}(\BS{t})=\bigl(\BS{\tilde{f}}^\trp(\BS{t})\tilde{\BS{\beta}}\bigr)^{-1}\BS{\tilde{f}}(\BS{t})$
where $\BS{\tilde{f}}(\BS{t})=\BS{Q}\BS{f}^\circ(\BS{t})$ and $\tilde{\BS{\beta}}=\big(\BS{Q}^\trp\big)^{-1}\BS{\beta}=\big(\beta_3,\beta_2,\beta_1 \big)^\trp$. Note that $\BS{\tilde{f}}(\BS{t})=(1,t_1,t_2)^\trp$.
Thus
\begin{eqnarray}
\BS{\tilde{f}}_{\BS{\tilde{\beta}}}(\BS{t})=\bigl(\beta_3+\beta_2t_1 +\beta_1t_2\bigr)^{-1}\,\big(1,t_1,t_2 \big)^\trp. \label{eq4-2}
\end{eqnarray}
Since (\ref{eq4-2}) coincides with that for a model with intercept  the  D-criterion is equivariant (see \cite{radloff2016invariance}) with respect to 
a one-to-one  transformation from  $\mathcal{T}$ to $\mathcal {Z}=[0,1]^2$ where 
 \begin{eqnarray}
 t_j\rightarrow z_j=\frac{1}{(1/a)-(1/b)}t_j-\frac{1/b}{(1/a)-(1/b)},j=1,2. \label{eq4-3}
 \end{eqnarray} 
For a given transformation matrix    \begin{eqnarray*}
  \BS{B}=\left(\small
  \begin{array}{ccc}
  1&0&0\\ \noalign{\medskip}
 \frac{-(1/b)}{(1/a)-(1/b)}&\frac{1}{(1/a)-(1/b)}&0\\ \noalign{\medskip}
 \frac{-(1/b)}{(1/a)-(1/b)}&0&\frac{1}{(1/a)-(1/b)} 
 \end{array}
 \right)
 \,\,\mbox{ with }\,\,
  \BS{B}^{-1}=\left(\small
  \begin{array}{ccc}
  1&0&0\\ \noalign{\medskip}
   \frac{1}{b}&\frac{1}{a}-\frac{1}{b}&0\\ \noalign{\medskip}
 \frac{1}{b}&0&\frac{1}{a}-\frac{1}{b} 
       \end{array}
       \right)
   \end{eqnarray*}
we have  $\BS{\tilde{\tilde{\beta}}}=\big(\BS{B}^\trp\big)^{-1}\tilde{\BS{ \beta}}=(\tilde{\tilde{\beta}}_0,\,\tilde{\tilde{\beta}}_{1},\,\tilde{\tilde{\beta}}_{2})^\trp$\, and hence\,   $\tilde{\tilde{\beta}}_0=\beta_3+(1/b)(\beta_1+\beta_2)$\,,\, $\tilde{\tilde{\beta}}_{1}=\beta_2((1/a)-(1/b))$\, and \,  $\tilde{\tilde{\beta}}_{2}=\beta_1((1/a)-(1/b))$. 
It follows that 
\begin{eqnarray}
\BS{\tilde{f}}_{\BS{\tilde{\tilde{\beta}}}}(\BS{z})=\BS{B}\BS{\tilde{f}}_{\BS{\tilde{\tilde{\beta}}}}\big(\BS{t}\big)=\bigl(\tilde{\tilde{\beta}}_0+\tilde{\tilde{\beta}}_{1}z_1+\,\tilde{\tilde{\beta}}_2 z_{2})^{-1}\,\bigl(1\,,\,z_1\,,\,z_{2}\bigr)^\trp, \BS{z}\in[0,1]^2. \label{eq4-4}
\end{eqnarray}
 
Let $\BS{M}(\BS{x},\BS{\beta})=\BS{f}_{\BS{\beta}}(\BS{x})\BS{f}_{\BS{\beta}}^\trp(\BS{x})$,\, $\BS{\tilde{M}}(\BS{t},\tilde{\BS{ \beta}})=\BS{\tilde{f}}_{\BS{\tilde{\beta}}}(\BS{t})\BS{\tilde{f}}_{\BS{\tilde{\beta}}}^\trp(\BS{t})$ and  $\BS{\tilde{\tilde{M}}}(\BS{z},\BS{\tilde{\tilde{\beta}}})=\BS{\tilde{f}}_{\BS{\tilde{\tilde{\beta}}}}(\BS{z})\BS{\tilde{f}}_{\BS{\tilde{\tilde{\beta}}}}^\trp(\BS{z})$ be the information matrices for the models which corresponding to (\ref{eq4-1}), (\ref{eq4-2}) and (\ref{eq4-4}), respectively. It is easily to observe that
\begin{eqnarray*}
\BS{M}(\BS{x},\BS{\beta}) =\BS{Q}^{-1}\BS{\tilde{M}}(\BS{t},\BS{\tilde \beta})\BS{Q}^{-1}=\BS{B}^{-1}\BS{Q}^{-1}\BS{\tilde{\tilde{M}}}(\BS{z},\BS{\tilde{\tilde{\beta}}})\BS{Q}^{-1}\BS{B}^{-1},
\end{eqnarray*}
thus the  derived D-optimal designs on $\mathcal{X}$,  $\mathcal{T}$ and  $\mathcal{Z}$, respectively are equivariant. According to the mapping of $\BS{x}$ to $\BS{t}$ in the line following (\ref{eq4-1}) and the mapping from $\BS{t}$ to $\BS{z}$ in (\ref{eq4-3}) each component is mapped separately: $x_j \to t_j \to z_j$ without permuting them. Therefore,  one modifies the direct one-to-one transformation $g: \mathcal {X} \rightarrow \mathcal {Z}$ where 
 \begin{eqnarray}
 x_j\rightarrow z_j=\frac{1/x_j}{(1/a)-(1/b)}-\frac{1/b}{(1/a)-(1/b)},j=1,2.\label{eq4-5}
 \end{eqnarray} 
Let  $\xi_{g}^{*}$ be a design defined on $\mathcal {Z}$ that assigns the weights $\xi(\BS{x})$ to the mapped support points $g(\BS{x}),\,\BS{x}\in \mathrm{supp}(\xi^*)$. In fact,  $\xi^*$ on $\mathcal{X}$  is locally D-optimal (at  $\BS{\beta}$) if and only if  $\xi_{g}^{*}$ on $\mathcal {Z}$ is locally D-optimal (at  $\BS{\tilde{\tilde{\beta}}}$). It is worth noting by transformation (\ref{eq4-5}) we obtain
 \begin{align*}
 &(b,b)^\trp\rightarrow (0,0)^\trp,\,\, (b,a)^\trp\rightarrow (1,0)^\trp,\\
 &(a,b)^\trp\rightarrow (0,1)^\trp,\,\, (a,a)^\trp\rightarrow (1,1)^\trp.
   \end{align*}
\begin{corollary}\label{theo4-1}
Consider  $\BS{f}(\BS{x})=\big(x_1,x_2,x_1x_2 \big)^\trp$  on  $\mathcal{X}=[a,b]^2,\,0<a<b$. Denote the vertices by $\BS{v}_1=(b,b)^\trp$,  $\BS{v}_2=(b,a)^\trp$, $\BS{v}_3=(a,b)^\trp$,  $\BS{v}_4=(a,a)^\trp$. Let $\BS{\beta}=(\beta_1,\beta_2,\beta_3)^\trp$ be a parameter point. Then the unique locally D-optimal design $\xi^*$ (at $\BS{\beta}$)  is as follows. 
\begin{enumerate}[(i)]
\item  
If \ $\beta_3^2+ \frac{1}{b^2}(\beta_1^2+\beta_2^2)+( \frac{1}{b^2}- \frac{1}{a^2}+ \frac{2}{a\,b})\beta_1\beta_2+ \frac{2}{b}\beta_3(\beta_1+\beta_2)\le 0$ \ then $\xi^*$ assigns equal weights $1/3$ to $\BS{v}_1,\BS{v}_2,\BS{v}_3$.
\item  
If \  $ \beta_3^2+ \frac{1}{b^2}\beta_1^2+ \frac{1}{a^2}\beta_2^2+ \frac{2}{b}\beta_3\beta_1+\frac{2}{a}\beta_3\beta_2+(\frac{1}{b^2}+\frac{1}{a^2})\beta_1\beta_2\le 0$  \ \ then $\xi^*$ assigns equal weights $1/3$ to $\BS{v}_1,\BS{v}_2,\BS{v}_4$.
\item 
If \  $ \beta_3^2+ \frac{1}{b^2}\beta_2^2+ \frac{1}{a^2}\beta_1^2+ \frac{2}{b}\beta_3\beta_2+\frac{2}{a}\beta_3\beta_1+(\frac{1}{b^2}+\frac{1}{a^2})\beta_1\beta_2\le 0$  \ \ then $\xi^*$ assigns equal weights $1/3$ to $\BS{v}_1,\BS{v}_3,\BS{v}_4$. 
\item  
If \ $\beta_3^2+\frac{1}{a^2}(\beta_1^2+\beta_2^2)+(\frac{1}{a^2}-\frac{1}{b^2}+\frac{2}{ab}) \beta_1\beta_2+\frac{2}{a}\beta_3(\beta_1+\beta_2)\le 0$ \ then $\xi^*$ assigns equal weights $1/3$ to $\BS{v}_2,\BS{v}_3,\BS{v}_4$. 
 \item  
 If none of the cases $(i)$ -- $(iv)$  applies then $\xi^*$ is supported by the four vertices 
 \[
 \xi^*=\left(\begin{array}{cccc} \BS{v}_1 &\BS{v}_2 & \BS{v}_3 &\BS{v}_4\\ [.5ex] 
\omega_1^*&\omega_2^* &\omega_3^*&\omega_4^*\end{array}\right),\,\,\mbox{ where $ \omega_\ell^*>0\  (1\le \ell\le4),\  \sum_{\ell=1}^4\omega_\ell^*=1$.}
\]
\end{enumerate}
\end{corollary}

\begin{proof}
The regression vector $\BS{\tilde{f}}_{\BS{\tilde{\tilde{\beta}}}}(\BS{z})$ given by (\ref{eq4-4}) coincides with that for the two-factor gamma model with intercept on $\mathcal{Z}=[0,1]^2$ whose intensity function is defined as $u_{\BS{\tilde{\tilde{\beta}}}}(\BS{z})=(\tilde{\tilde{\beta}}_0+\tilde{\tilde{\beta}}_{1}z_1+\,\tilde{\tilde{\beta}}_2 z_{2})^{-2}$ for all $\BS{z}\in \mathcal{Z}$.
Denote 
\begin{align*}
c_1&=u_{\BS{\tilde{\tilde{\beta}}}}((0,0)^\trp)=\tilde{\tilde{\beta}}_0^{-2}=({\beta_3+\frac{1}{b}(\beta_1+\beta_2}))^{-2},\\
c_2&=u_{\BS{\tilde{\tilde{\beta}}}}((1,0)^\trp)=(\tilde{\tilde{\beta}}_0+\tilde{\tilde{\beta}}_1 )^{-2}=(\beta_3+\beta_1\frac{1}{b}+\beta_2\frac{1}{a})^{-2},\\
c_3&=u_{\BS{\tilde{\tilde{\beta}}}}((0,1)^\trp)=(\tilde{\tilde{\beta}}_0+\tilde{\tilde{\beta}}_2 )^{-2}=(\beta_3+\beta_1\frac{1}{a}+\beta_2\frac{1}{b})^{-2},\\
c_4&=u_{\BS{\tilde{\tilde{\beta}}}}((1,1)^\trp)=(\tilde{\tilde{\beta}}_0+\tilde{\tilde{\beta}}_1+\tilde{\tilde{\beta}}_2 )^{-2}=(\beta_3+\frac{1}{a}(\beta_1+\beta_2))^{-2}.
\end{align*}
Let $h, i, j, k \in \{1,2,3,4\}$ are pairwise distinct  such that $c_k=\min\{c_1,c_2,c_3,c_4\}$ then it follows  from Theorem 4.2 in \cite{Gaffke2018} that if  \ $c_{k}^{-1}\,\ge c_{h}^{-1}+c_{i}^{-1}+c_{j}^{-1}$ \ then  $\xi^*$ is a three-point design supported by the three vertices $\BS{v}_h$, $\BS{v}_i$, $\BS{v}_j$, with equal weights $1/3$. Hence, straightforward computations show that the  condition in case $(i)$ of the corollary  is equivalent to $c_{4}^{-1}\,\ge c_{1}^{-1}+c_{2}^{-1}+c_{3}^{-1}$. Analogous verifying is obtained for other cases. By Remark \ref{rem-1}  the four-point design with positive weights in case $(v)$   applies implicitly if non of the conditions  $(i)$ -- $(iv)$ of saturated designs  is satisfied at a given $\BS{\beta}$. 
\end{proof}

It is noted that, the optimality conditions $(i)$--$(iv)$ provided by Corollary  \ref{theo4-1} depend on  the values of  $a$ and $b$.  The D-optimality  might be achieved or declined  by changing  the values of $a$ and $b$. To see that, more specifically, let $a=1$ and $b=2$, i.e., the experimental region is  $\mathcal{X}=[1,2]^2$ and define $\gamma_1=\beta_1/\beta_3$ and $\gamma_2=\beta_2/\beta_3, \beta_3 \neq 0$. Here, the parameter space   which is depicted in Panel (a) of Figure \ref{fig:opt2} is characterized by $\gamma_2+\gamma_1>-1$, $2\,\gamma_2+\gamma_1>-2$ and $\gamma_2+2\,\gamma_1>-2$.  It is observed that from Panel (a) of  Figure  \ref{fig:opt2} that the design given by part $(i)$ of Corollary  \ref{theo4-1} is not locally D-optimal at any parameter point belongs to the space of model parameters. In other words,  condition $(i)$,  $\frac{1}{4}(\gamma_1^2+\gamma_2^2)+\frac{1}{4}\gamma_1\gamma_2+ \gamma_1+\gamma_2\le-1$, can not be satisfied. \par
Let us consider another experimental region with a higher length  by fixing $a=1$ and taking $b=4$, i.e., $\mathcal{X}=[1,4]^2$.  The parameter space   which is depicted in Panel (b) of Figure \ref{fig:opt2} is characterized by $\gamma_2+\gamma_1>-1$,  $4\,\gamma_2+\gamma_1>-4$ and $\gamma_2+4\,\gamma_1>-4$. In this case all designs given by Corollary \ref{theo4-1} are locally D-optimal at particular values of $\gamma_2$ and $\gamma_1$ as it is observed from the figure.  It is obvious that along the diagonal dashed line ($\gamma_2=\gamma_1$) there exist  at most three different types of locally D-optimal designs.\par
\begin{figure}[H]
    \centering
    \subfloat[]{{\includegraphics[width=7cm, height=6cm]{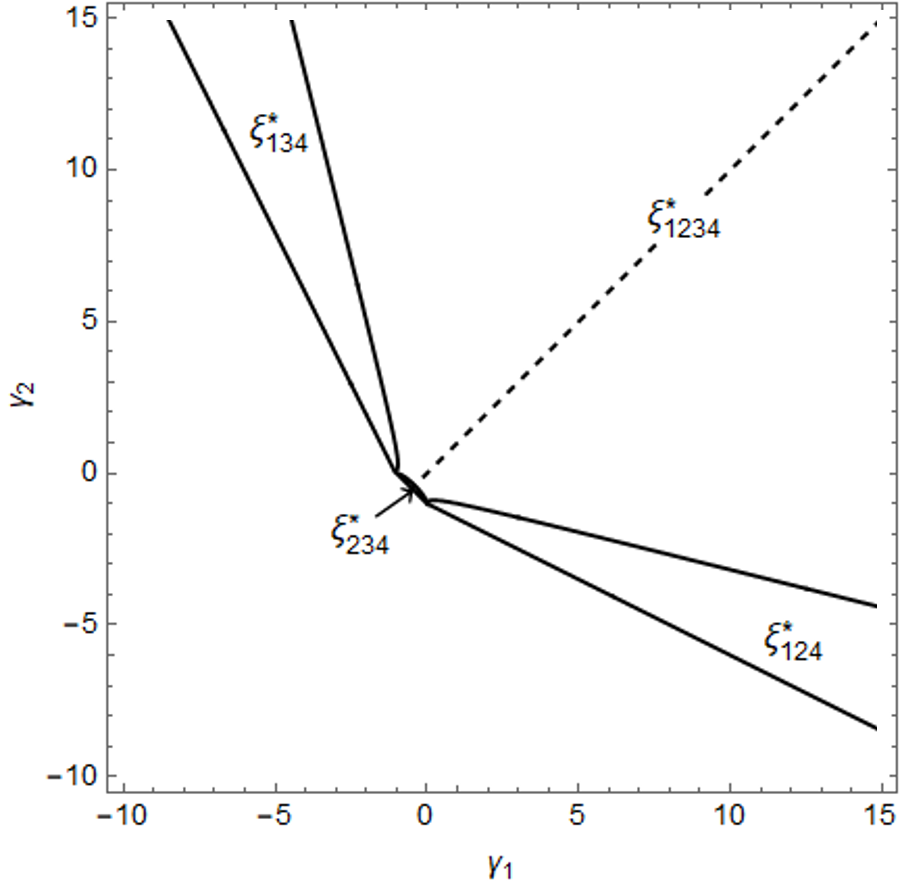} }}
    \qquad
    \subfloat[]{{\includegraphics[width=7cm, height=6cm]{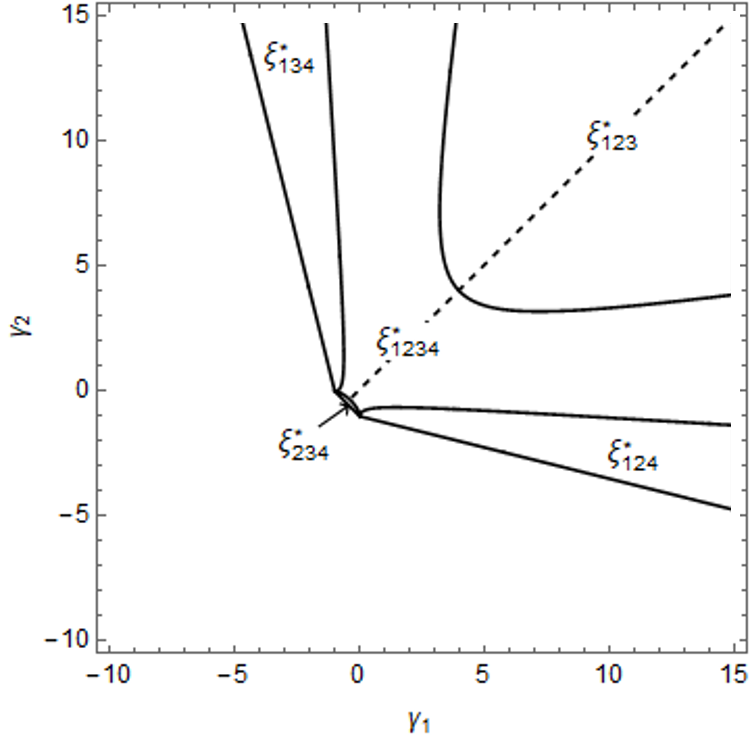} }}
    \caption{Dependence of locally D-optimal designs on $\gamma_1=\beta_1/\beta_3$ and $\gamma_2=\beta_2/\beta_3$ where for Panel (a) $\mathcal{X}=[1,2]^2$ and  for Panel (b) $\mathcal{X}=[1,4]^2$.  The diagonal dashed line represents the case $\gamma_2=\gamma_1$.  Note that 
    $\mathrm{supp}(\xi^*_{ijk})=\{\BS{v}_i,\BS{v}_j,\BS{v}_k\}\subset\{\BS{v}_1,\BS{v}_2,\BS{v}_3,\BS{v}_4\}$ and  $\mathrm{supp}(\xi^*_{1234})=\{\BS{v}_1,\BS{v}_2,\BS{v}_3,\BS{v}_4\}$. }. 
    \label{fig:opt2}
\end{figure}
 For arbitrary values of  $a$ and $b$, $0<a<b$  let us restrict to case $\gamma_2=\gamma_1=\gamma$, i.e., $\beta_1=\beta_2=\beta,\beta_3\neq 0$ and the next corollary is immediate. 
\begin{corollary} \label{lem-int}
Consider $\BS{f}(\BS{x})=(x_1,x_2,x_1x_2)^\trp$ on an arbitrary square $\mathcal{X}=[a,b]^2,\,0<a<b$ in the positive quadrant. Let $\beta_1=\beta_2=\beta$ and $\beta_3\neq 0$. Define   $\gamma=\frac{\beta}{\beta_3}$. Then the locally  D-optimal design $\xi^*$  (at $\BS{\beta}$)  is as follows.
\begin{enumerate}[(i)]
\item If $-\frac{a}{2}<\gamma\le -\frac{ab}{3b-a}$, then $\xi^*$ assigns equal weights $1/3$ to  $\BS{v}_2,\BS{v}_3$, $\BS{v}_4$. 
\item  If  $b-3a>0$ and  $\gamma\ge\frac{ab}{b-3a}$, then $\xi^*$ assigns equal weights $1/3$ to  $\BS{v}_1,\BS{v}_2,\BS{v}_3$. 
\item  If  $b-3a>0$ and $-\frac{ab}{3b-a}<\gamma < \frac{ab}{b-3a}$ then the design $\xi^*$   is supported by $\BS{v}_1$, $\BS{v}_2,\BS{v}_3$, $\BS{v}_4$. 
 The optimal weights are given by 
 \[
\omega_1^*=\frac{ab-(a-3b)\gamma}{4b(a+2\gamma)},\ \ 
\omega_2^*=\omega_3^*=\frac{\big(ab+(a+b)\gamma\big)^2}{4ab(b+2\gamma)(a+2\gamma)},\ \ 
\omega_4^*=\frac{ab-(b-3a)\gamma}{4a(b+2\gamma)}.
\]
\end{enumerate}
 \end{corollary}

\begin{proof}
Consider the experimental region $\mathcal{X}=[a,b]^2,\,0<a<b$. By  assumption $\beta_1=\beta_2=\beta,\,\beta_3\neq0$ the range of $\gamma=\frac{\beta}{\beta_3}$ is given by $(-a/2,\infty)$. Assumption $b-3a>0$ implies that  $-\frac{a}{2}<-\frac{ab}{3b-a}<\frac{ab}{b-3a}$. Employing  Corollary \ref{theo4-1} shows the following.  Both conditions of parts ($\rmnum{2}$) and ($\rmnum{3}$) of Corollary \ref{theo4-1} are not fulfilled by any parameter point thus the corresponding designs are not D-optimal. In contrast, the design  $\xi^*$ in ($\rmnum{1}$) of  Corollary \ref{lem-int} is locally D-optimal if the condition of  part $(v)$ of Corollary \ref{theo4-1} holds true. That condition is equivalent to 
 \[
  (3b^2+2ab-a^2)\gamma^2+4 ab^2\gamma+a^2b^2 \le0.
  \]
The l.h.s. of above inequality is polynomial in $\gamma$ of degree 2 and  thus the inequality is fulfilled  by $-\frac{a}{2}<\gamma\le -\frac{ab}{3b-a}$.\\*
Again, the design  $\xi^*$ in ($\rmnum{2}$)  is locally D-optimal if  the condition of  part $(i)$ of Corollary \ref{theo4-1} holds true.  That condition is equivalent to 
 \[
  (3a^2+2ab-b^2)\gamma^2+4 a^2b\gamma+a^2b^2 \le0.
  \]
The l.h.s. of above inequality is polynomial in $\gamma$ of degree 2 and  thus the inequality is fulfilled  by $\gamma\ge\frac{ab}{b-3a}$ if  $b-3a>0$.\\* 
The four-point design given   in ($\rmnum{3}$) has  positive  weights on $-\frac{ab}{3b-a}<\gamma < \frac{ab}{b-3a}$ if  $b-3a>0$ and hence it is implicitly locally D-optimal by  Remark \ref{rem-1}.   
  \end{proof}
\begin{remark} One should note that from Corollary \ref{lem-int} when $\beta=0$ the uniform design on the vertices $\BS{v}_1, \BS{v}_2, \BS{v}_3, \BS{v}_4$ is locally D-optimal.
\end{remark}
\section{Design efficiency}

The D-optimal design for gamma models depends on a given value of the parameter $\BS{\beta}$.  Misspecified values may lead to a poor performance of the locally optimal design. From our results the  designs are locally D-optimal at a specific subregion of the parameter space.  In this section we  discuss  the potential benefits of the derived designs, in particular, the D-optimal designs from Theorem  \ref{The5-5} for a gamma model without interaction and from Corollary \ref{lem-int} for a gamma model with interaction. Our objective is to examine the overall performance of some of  the locally D-optimal designs. The overall performance of any design $\xi$ is described by its D-efficiencies, as a function of $\BS{\beta}$,  
\begin{eqnarray}
\mathrm{Eff}(\xi,\BS{\beta})=\left(\frac{\det\BS{M}(\xi, \BS{\beta})}{\det\BS{M}(\xi_{\BS{\beta}}^*,\BS{\beta})}\right)^{1/3} \label{eq6.1}
\end{eqnarray}
where $\xi_{\BS{\beta}}^*$ denotes the locally D-optimal design at $\BS{\beta}$.

\hspace{-4ex} {\bf Example 1.} In the situation of  Theorem  \ref{The5-5} the experimental region is given by $\mathcal{X}=[1,2]^3$. We restrict only to the case $\beta_1>0$, $\beta_2=\beta_3=\beta$ and hence we utilize the ratio $\gamma=\beta/\beta_1$ with range $(-1/4,\infty)$. 
Our interest is in the saturated and equally weighted designs  $\xi_1$ and $\xi_2$ where $\mathrm{supp}(\xi_1)=\{\BS{v}_2,\BS{v}_3,\BS{v}_4\}$ and  $\mathrm{supp}(\xi_2)=\{\BS{v}_3,\BS{v}_4,\BS{v}_5\}$  which by Theorem \ref{The5-5}  are locally D-optimal at  $\gamma \ge 1/5$ and  $\gamma \in (-1/4,-5/23]$, respectively. In particular, $\xi_1$ and  $\xi_2$ are robust against misspecified parameter values in their respective subregions. Additionally, for $\gamma\in(-5/23,1/5)$ we consider the locally D-optimal designs of type $\xi_3(\gamma)$  given by the theorem. Note that  $\mathrm{supp}(\xi_3(\gamma))=\{\BS{v}_2,\BS{v}_3,\BS{v}_4,\BS{v}_5\}$ and the weights depend on $\gamma$.  \par
   
To employ (\ref{eq6.1}) we put  $\xi_{\BS{\beta}}^*=\xi_1$ if $\gamma \ge 1/5$, $\xi_{\BS{\beta}}^*=\xi_2$ if $\gamma \in (-1/4,-5/23]$ and $\xi_{\BS{\beta}}^*=\xi_3(\gamma)$ if $\gamma\in(-5/23,1/5)$. We select for examination the designs $\xi_1$, $\xi_2$, $\xi_3(-1/7)$. Moreover, as  natural competitors we select various uniform designs supported by specific vertices. That is $\xi_4$ with support $\{1,2\}^3$ and the two half-fractional designs $\xi_5$ and $\xi_6$ supported by $\{\BS{v}_1,\BS{v}_5,\BS{v}_6,\BS{v}_7\}$  and   $\{\BS{v}_2,\BS{v}_3,\BS{v}_4,\BS{v}_8\}$, respectively. Additionally, we consider  $\xi_7$ which assigns uniform weights to the grid $\{1,1.5,2\}^3$.

In Panel (a) of Figure \ref{fig:1}, the D-efficiencies of the  designs $\xi_1$, $\xi_2$, $\xi_3(-1/7)$, $\xi_4$, $\xi_5$, $\xi_6$ and $\xi_7$  are depicted.  The efficiencies of $\xi_1$ and $\xi_2$ are, of course, equal to $1$ in their optimality subregions $\gamma \in [1/5,\infty)$ and $\gamma \in (-1/4,-5/23]$, respectively. However, for $\gamma$ outside but  fairly close to the respective  optimality subregion both designs perform quite well; the efficiencies of $\xi_1$ and $\xi_2$ are greater than $0.80$ for $-0.15\le \gamma <1/5$ and $-1/4< \gamma \le-0.28$, respectively. However, their efficiencies decrease towards zero when $\gamma$ moves away from the respective  optimality subregion. So, the overall performance of $\xi_1$ and $\xi_2$ cannot be regarded as satisfactory. The design $\xi_{3}(-1/7)$, though locally D-optimal only at $\gamma=-1/7$, does show a more satisfactory overall performance with efficiencies range between $0.8585$ and $1$. The   efficiencies of  the half-fractional design $\xi_6$ are greater than $0.80$ only for $\gamma>-0.049$, otherwise the  efficiencies decrease towards zero. The  design $\xi_4$  turns out to be uniformly worse than $\xi_{3}(-1/7)$ and its efficiencies range between $0.5768$ and $0.7615$. The worst performance is shown by the designs $\xi_5$ and $\xi_7$.\\*

\hspace{-4ex} {\bf Example 2.} In the situation of  Corollary \ref{lem-int} we consider the experimental region  $\mathcal{X}=[1,4]^2$ where condition $b-3a>0$ is satisfied. The vertices are denoted by   $\BS{v}_1=\big(4,4\big)^\trp$, $\BS{v}_2=\big(4,1\big)^\trp$,  $\BS{v}_3=\big(1,4\big)^\trp$, $\BS{v}_4=\big(1,1\big)^\trp$. We restrict to $\beta_3 \neq 0$, $\beta_1=\beta_2=\beta$, and the range of $\gamma=\beta/\beta_3$ is $ (-1/2, \infty)$. In analogy to {\bf Example 1}  denote by $\xi_1$ and $\xi_2$ the saturated and equally weighted designs with support $\{\BS{v}_1,\BS{v}_2,\BS{v}_3\}$ and $\{\BS{v}_2,\BS{v}_3,\BS{v}_4\}$, respectively. By the corollary $\xi_1$ and $\xi_2$  are locally D-optimal at $\gamma \ge4$ and $\gamma \in (-1/2,-4/11]$, restrictively. Denote by $\xi_3(\gamma)$ the design given in part $(ii)$ of  Corollary \ref{lem-int} which is locally D-optimal at $\gamma \in (-4/11, 4)$. Note that  from (\ref{eq6.1}) we put  $\xi_{\BS{\beta}}^*=\xi_1$ if $\gamma \ge 4$, $\xi_{\BS{\beta}}^*=\xi_2$ if $\gamma \in (-1/2,-4/11]$ and $\xi_{\BS{\beta}}^*=\xi_3(\gamma)$ if $\gamma\in (-4/11, 4)$.  For examination we select $\xi_1$, $\xi_2$, $\xi_3(0)$. As a natural competitor we select $\xi_4$ that assigns uniform weights to the grid $\{1,2.5,4\}^2$. The efficiencies  are depicted in Panel (b) of Figure \ref{fig:1}. We observe that the performance of  $\xi_1$ and $\xi_2$ is similar to that of the corresponding designs in {\bf Example 1}.  Moreover, the design $\xi(0)$ show a more satisfactory overall performance. The efficiencies of $\xi_4$ vary between $0.77$ and $0.83$ for $\gamma > -4/11$.   

\begin{figure}[H]
    \centering
    \subfloat[{\bf Example 1}. The considered interval is $-1/4<\gamma \le 1$.]{{\includegraphics[width=7cm, height=6cm]{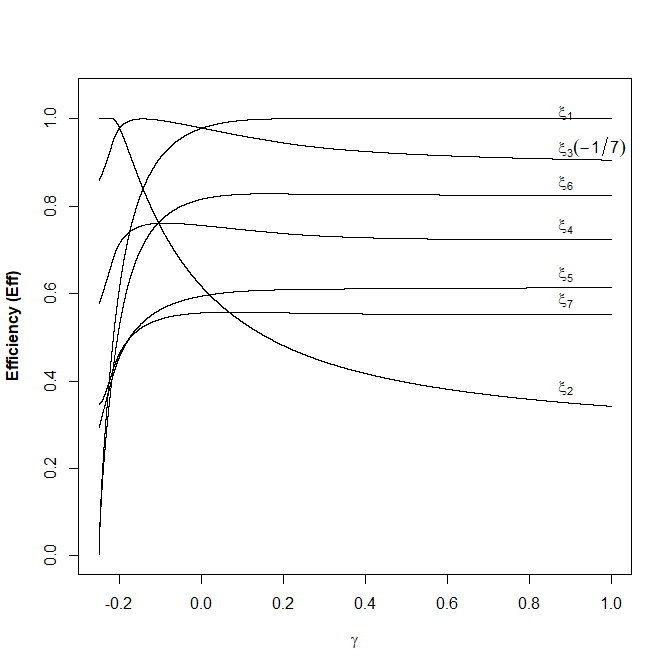} }}
    \qquad
    \subfloat[{\bf Example 2}. The considered interval is $-1/2<\gamma \le 5$.]{{\includegraphics[width=7cm, height=6cm]{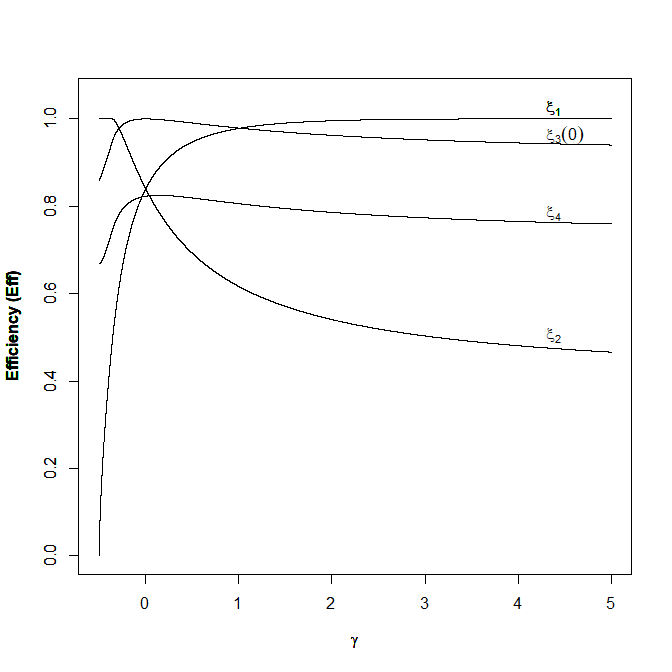} }}
    \caption{ D-efficiencies from (\ref{eq6.1}) of particular designs relative to the optimal designs at their optimality subregions under gamma models without intercept.}%
    \label{fig:1}
\end{figure}

\section{Conclusion}
In the current paper we considered gamma models without intercept for which locally D- and A-optimal designs have been developed. The positivity of the expected means entails a positive linear predictor whereas absence of the intercept term requires an experimental region  which is not containing the origin point $\BS{0}$.   The information matrix for the non-intercept gamma model is invariant w.r.t. simultaneously scaling of $\BS{x}$ or $\BS{\beta}$. In this context, we utilized different approaches to derive the locally optimal designs. Sets of D- and A-optimal designs were derived on a non-compact experimental region. On the other hand,  a transformation to models that are having intercept were employed  on  a two-factor model without or with interaction as in Theorem \ref{The3-2} or Corollary \ref{theo4-1}, respectively.   This approach simplified the optimality problem and thus such known results were applied.  Moreover, the complexity of applying The  Equivalence Theorem as in Theorem  \ref{The5-5}  implicated the optimality  problem to solve a system of inequalities analytically or by employing  computer algebra. In contrast, the transformation approach, of course, can be used for the case in Theorem \ref{The5-5} and thus  according to  Remark \ref{rem1}, the three-factor model without intercept on $\mathcal{X}=[1,2]^2$ can be transformed to a model with intercept on\, $\mathcal{T}={\rm Conv}\bigl\{(1/2,1)^\trp,(1,1/2)^\trp,(1/2,1/2)^\trp,(2,1)^\trp,(1,2)^\trp,(2,2)^\trp\bigr\}$. Rescaling $\mathcal{T}$ yields  $\mathcal{Z}={\rm Conv}\bigl\{ (0,1/3)^\trp,(1/3,0)^\trp,(0,0)^\trp,(1,1/3)^\trp,(1/3,1)^\trp,(1,1)^\trp \bigr\}$.  Consequently, the linear predictor is  reparameterized  as $\tilde \beta_0+\tilde\beta_{1} z_1+\tilde\beta_{2} z_2$ where  $(z_1,z_2)^\trp\in {\mathcal{Z}}$ and   $\tilde \beta_0=\beta_1+(1/2)(\beta_2+\beta_3)$,  $\tilde\beta_{1}=(3/2)\beta_2$, $\tilde\beta_{2}=(3/2)\beta_3$.\par

In many applied aspects, the log-link function is considered as a main alternative to the canonical one (see \cite{Kilian2002ACO},\cite{wenig2009costs},\cite{10.1093/intqhc/mzr010},\cite{McCrone2005},\cite{Montez-Rath2006}).  In that case the intensity function $u(\BS{x},\BS{\beta})=1$ and thus the information matrix under gamma models is equivalent to that under ordinary regression models. For that reason, the optimal designs for a gamma model are identical to those for an ordinary regression model with similar linear predictor.  In  \cite{hardin2018generalized} gamma models were fitted considering various link functions, for example;  the Box-Cox family of link functions that is given by 
\begin{equation}
\BS{f}^\trp(\BS{x})\BS{\beta}=\left\{\begin{array}{ll} \bigl(\mu^{\lambda}-1\bigr)/\lambda  & (\lambda \neq 0)  \\ \log \mu & (\lambda=0) \end{array}\right.
\end{equation}
which involves the log-link  at $\lambda = 0$ (see \cite{atkinson2015designs}). The intensity function is thus defined  as 
\begin{equation}
u(\BS{x},\lambda\BS{\beta})=\big(\lambda \BS{f}^\trp(\BS{x})\BS{\beta}+1\big)^{-2}, \BS{x}\in \mathcal{X}.
\end{equation}
Here, the positivity condition (\ref{eq2-2}) of the expected mean $\mu=E(y)$ of a gamma distribution is modified to $\lambda \BS{f}^\trp(\BS{x})\BS{\beta}>-1$ for all $\BS{x}\in \mathcal{X}$. Therefore,  for a gamma model without intercept the experimental region might be considered as $ \mathcal{X}=[0,1]^\nu$. As an example,  consider $\BS{f}(\BS{x})=(x_1,x_2)^\trp$ on  $\mathcal{X}=[0,1]^2$ with  vertices  $\BS{v}_1=(0,0)^\trp$, $\BS{v}_2=(1,0)^\trp$, $\BS{v}_3=(0,1)^\trp$, $\BS{v}_4=(1,1)^\trp$.  Let $u_{k}=u(\BS{v}_k, \lambda\BS{\beta})$ for all $(1 \le k \le 4)$. The Equivalence Theorem (Theorem \ref{theo2-1}, part (a)) approves the D-optimality of the design $\xi^*$ which assigns equal weights $1/2$ to the vertices $\BS{v}_2$ and $\BS{v}_3$  at the point  $\lambda \BS{\beta}$. This result might be extended for  a multiple-factor model  as in Theorem  \ref{theo5.0.1}. However, the expression $\lambda \BS{f}^\trp(\BS{x})\BS{\beta}+1$  could be viewed as  a linear predictor of a gamma model with known intercept.  Adopting the Box-Cox family as a class of link functions for gamma models could be a topic of future research.

\end{document}